\documentclass[11pt]{amsart}
\usepackage{hyperref}
\usepackage{amsmath}
\usepackage{amssymb}
\usepackage{amsfonts}
\usepackage{amsthm}
\usepackage{latexsym}
\usepackage{esint}
\usepackage{mathtools}
\usepackage{mathrsfs}
\usepackage{graphicx}
\usepackage{bbm}
\usepackage{todonotes}
\usepackage{color}
\usepackage{bm}
\usepackage[top=1in, bottom=1.25in, left=1.10in, right=1.10in]{geometry}

\newtheorem{theorem}{Theorem}[section]
\newtheorem{lemma}[theorem]{Lemma}
\newtheorem{proposition}[theorem]{Proposition}
\newtheorem{corollary}[theorem]{Corollary}

\theoremstyle{definition}
\newtheorem{definition}[theorem]{Definition}

\theoremstyle{remark}
\newtheorem{remark}[theorem]{Remark}

\numberwithin{equation}{section}

\newcommand{\bu}{\mathbf{u}}

\newcommand{\dx}{\, \mathrm{d} \mathbf{x}}

\begin{document}

\title[Incompressible quasineutral  limit of the stochastic Navier--Stokes--Poisson system]
{The combined incompressible quasineutral  limit of the stochastic Navier--Stokes--Poisson system}%
%\author{Paolo Antonelli}
%\address{GSSI - Gran Sasso Science Institute \\
%Viale F. Crispi, 7 \\
%67100 L’Aquila, Italy.}
%\email{paolo.antonelli@gssi.infn.it}

%    Information for first author
\author{Donatella Donatelli}
%    Address of record for the research reported here
\address{Department of Information Engineering, Computer Science and Mathematics \\
University of L'Aquila\\
67100 L'Aquila, Italy}
\email{donatella.donatelli@univaq.it}
%\thanks{ }

%%    Information for second author
%\author{Pierangelo Marcati}
%\address{GSSI - Gran Sasso Science Institute \\
%Viale F. Crispi, 7 \\
%67100 L’Aquila, Italy.}
%\email{pierangelo.marcati@univaq.it, \, pierangelo.marcati@gssi.infn.it}
%%\thanks{Support information for the second author.}

%    Information for second author
\author{Prince Romeo Mensah}
\address{GSSI - Gran Sasso Science Institute \\
Viale F. Crispi, 7 \\
67100 L'Aquila, Italy.}
\email{romeo.mensah@gssi.it}
%\thanks{Support information for the second author.}

%    General info
%\subjclass[2010]{Primary 35R60, 35Q35; Secondary  76N10, 76M45}

\date{\today}

%\dedicatory{This paper is dedicated to our advisors.}

%\keywords{Isentropic flows, Stochastic compressible fluid, Navier--Stokes, Mach number, Rossby number, Martingale solution}

\begin{abstract}
This paper deals with the combined incompressible quasineutral limit of the  weak martingale solution of the   compressible Navier--Stokes--Poisson system perturbed by a stochastic forcing term in the whole space. In the framework of  ill prepared initial data, we show the convergence in law to a weak martingale solution of a stochastic incompressible Navier--Stokes system. The result holds true for any arbitrary nonlinear forcing term with suitable growth. The proof is based on the analysis of acoustic waves but since we are dealing with a stochastic partial differential equation, the existing deterministic tools for treating this second-order equation breakdown. Although this might seem as a minor modification,  to  handle the acoustic waves in the stochastic compressible Navier--Stokes system, we produce  suitable dispersive estimate for  first-order system of equations, which are an added value to the existing theory.  As a by-product of this dispersive estimate analysis,  we are also able to prove a convergence result in the case of the zero-electron-mass limit for a stochastic  fluid dynamical plasma model.
\end{abstract}

\maketitle

\section{Introduction}
\label{sec:intro}
This paper deals with  a singular limit result for the following  stochastically-forced compressible Navier--Stokes--Poisson system 
\begin{align}
&\mathrm{d} \varrho^\varepsilon + \mathrm{div}(\varrho^\varepsilon \mathbf{u}^\varepsilon) \, \mathrm{d}t =0, 
\label{contEqBa}
\\
&\mathrm{d}(\varrho^\varepsilon \mathbf{u}^\varepsilon) +
\big[
\mathrm{div} (\varrho^\varepsilon \mathbf{u}^\varepsilon \otimes \mathbf{u}^\varepsilon) + \varepsilon^{-2}\nabla   (\varrho^\varepsilon)^\gamma \big]\, \mathrm{d}t 
=\big[ \nu_1 \Delta \mathbf{u}^\varepsilon + (\nu_2 + \nu_1)\nabla \mathrm{div} \mathbf{u}^\varepsilon
\nonumber\\
&\qquad\qquad\qquad\qquad\qquad\quad\qquad \qquad + \varepsilon^{-2}{\varrho^\varepsilon} \nabla V^\varepsilon \big] \, \mathrm{d}t + \mathbf{G}(\varrho^\varepsilon,  \varrho^\varepsilon\mathbf{u}^\varepsilon) \, \mathrm{d}W,
\label{momEqBa}
\\
&\lambda^2\Delta V^\varepsilon = \varrho^\varepsilon -1 
\label{elecFieldBa}
\end{align}
which is a simplified model (for instance, the temperature equation is not taken into account) to describe the dynamics of a plasma where the compressible electron fluid interacts with its own electric field against a constant charged ion background.\\
In the model \eqref{contEqBa}--\eqref{elecFieldBa}, ${\nu}_1,{\nu}_2\geq0$ are the viscosity constants,  $\varepsilon>0$ is the Mach number, $\lambda$ is the Debye length and the adiabatic exponent is $\gamma>\frac{3}{2}$. 
The  independent variables $(t,x,\omega)$ are contained in the random spacetime cylinder $[0,T]\times\mathbb{R}^3 \times \Omega$ where $T>0$ is a fixed constant and $\Omega$ is the sample space of a probability space $(\Omega, \mathscr{F}, \mathbb{P})$.  
On the other hand, the dependent variables are the charged density $\varrho(t,x)\in \mathbb{R}_{\geq0}=[0,\infty)$, the velocity vector $\mathbf{u}(t,x)\in \mathbb{R}^3$, the electrostatic potential $V(t,x)\in \mathbb{R}_{\geq0}$, and to enforce stochasticity, the cylindrical Wiener process $W(t,\omega)$. Since the system of equations \eqref{contEqBa}--\eqref{elecFieldBa} are mutually coupled, for simplicity, we enforce randomness throughout the system by way of the Wiener process. Nevertheless, the analysis to be performed remains valid if one assumes that all the dependent variables depend on the random parameter.
\\
Since the problem \eqref{contEqBa}--\eqref{elecFieldBa} is posed on the entire Euclidean space $\mathbb{R}^3$, we complement it with  the far-reach condition
\begin{align}
\varrho \rightarrow1,  \quad \vert  \mathbf{u}\vert \rightarrow 0, \quad \text{as} \quad \vert x\vert \rightarrow \infty
\end{align}
to describe the behaviour of the density and velocity field as they approach spatial infinity.
\\
Recently, the existence of a weak martingale solution  of the system \eqref{contEqB}--\eqref{elecFieldB} was carried out in \cite{donatelli2020dissipative} for a fixed Mach number $\varepsilon>0$ and fixed Debye length $\lambda>0$. In analogy with the purely fluid system studied
in  \cite{breit2018stoch, breit2016stochastic, mensah2016existence, mensah2019theses, smith2015random}, the proof of the result in \cite{donatelli2020dissipative} relies on the multi-layer approximation scheme initiated by Lions \cite{lions1998mathematical} and Feireisl, Novotn\'y and Petzeltov\'a \cite{feireisl2001existence} with suitable adaptation due to the presence of the Poisson equation \eqref{elecFieldBa}.
\\
We recall that the Mach number $\varepsilon$ is given by the ratio of the characteristic fluid velocity and  the sound speed and it is related with the fluid compressibility. In  many real world phenomena,  the fluid velocities are smaller compared to the sound speed, therefore, it is of interest to consider small values of $\varepsilon$ and to perform the related asymptotic analysis.  In the physical regime of small Mach number, it is observed that the pressure is unable to generate density variations, therefore, the asymptotic dynamics is described by the   incompressible physical state.
On the other hand, the Debye length $\lambda$ which is the coupling constant between the Navier--Stokes equations \eqref{contEqBa}-\eqref{momEqBa} and the Poisson equation \eqref{elecFieldBa}  is a characteristic physical parameter related to the phenomenon of the so-called ``Debye shielding'' \cite{Goldston1995} and in plasma physics, represents the distance over which the usual Coulomb field is killed off exponentially by the polarization of the plasma.  In many cases, the Debye length is very small compared to the macroscopic length, so it makes sense to consider the  limit  $\lambda\to 0$. This type of limit is called a  quasineutral limit since the charge density almost vanishes identically. If we consider the Poisson equation \eqref{elecFieldBa}, we can observe that under this regime, the particle density is constrained to be close to the background density which is one in our case. Hence, we can also conclude  in this case that the asymptotic state is the incompressible one.\\
Because of the previous considerations, it makes sense to study the {\em combined incompressible, quasineutral limit}. In this paper, we set\\
\begin{equation}
\lambda^2=\varepsilon^{\beta}, \qquad \beta>0,
\label{valuebeta}
\end{equation}
where the values of $\beta$ will be determined later on and we perform the limit, as $\varepsilon \to 0$, of the following system,\\
\begin{align}
&\mathrm{d} \varrho^\varepsilon + \mathrm{div}(\varrho^\varepsilon \mathbf{u}^\varepsilon) \, \mathrm{d}t =0, 
\label{contEqB}
\\
&\mathrm{d}(\varrho^\varepsilon \mathbf{u}^\varepsilon) +
\big[
\mathrm{div} (\varrho^\varepsilon \mathbf{u}^\varepsilon \otimes \mathbf{u}^\varepsilon) + \varepsilon^{-2}\nabla   (\varrho^\varepsilon)^\gamma \big]\, \mathrm{d}t 
=\big[ \nu_1 \Delta \mathbf{u}^\varepsilon + (\nu_2 + \nu_1)\nabla \mathrm{div} \mathbf{u}^\varepsilon
\nonumber\\
&\qquad\qquad\qquad\qquad\qquad\quad\qquad \qquad + \varepsilon^{-2}{\varrho^\varepsilon} \nabla V^\varepsilon \big] \, \mathrm{d}t + \mathbf{G}(\varrho^\varepsilon,  \varrho^\varepsilon\mathbf{u}^\varepsilon) \, \mathrm{d}W,
\label{momEqB}
\\
&\varepsilon^{\beta}\Delta V^\varepsilon = \varrho^\varepsilon -1 .
\label{elecFieldB}
\end{align}
\\
We will be able to show rigorously, in the setting of ill prepared initial data, that the target dynamics is given by the  stochastically-forced incompressible Navier--Stokes system.
\\ 
The study of singular limits for fluid-type systems such as \eqref{contEqB}--\eqref{elecFieldB} has seen extensive work done in the deterministic system starting with the pioneering work by Klainerman and Majda \cite{klainerman1981singular}. The zero Mach number limit result for the equations \eqref{contEqB}-\eqref{momEqB} have been extensively studied in the deterministic setting. We cannot give an exhaustive review on these results so we refer the reader to  \cite{desjardins1999incompressible, desjardins1999low, danchin2005low, donatelli2008euler, donatelli2010incompress, lions1998incompressible, schochet2005mathematical},  the references within them as well as papers citing them. In the stochastic setting, however, we only know of the results in 
\cite{breit2015incompressible, breit2018stoch, mensah2016existence, mensah2019theses}.
Nevertheless, related singular limit results in the stochastic setting also exist in \cite{mensah2018multi} where the two-dimensional stochastic incompressible Navier--Stokes equation is obtained from the three-dimensional stochastic compressible Navier--Stokes--Coriolis equation and in \cite{breitMen2019stochastic} where the stochastic compressible Euler equation is also derived from the stochastic compressible Navier--Stokes equation.
It is worth noting that by 
adapting the method of convex integration, which was introduced by De Lellis and Sz\'ekelyhidi \cite{delellis2012hprinciple} in the context of fluid dynamics, one can show that the initial value problem for the stochastic compressible Euler equation is actually ill-posed in the class of weak (distributional) solutions, see \cite{breit2020solvability}.
We also mention \cite{breit2015compressible} that constructs the relative energy inequality for the stochastic compressible Navier--Stokes system and apply this inequality to show weak-strong uniqueness results as well as to obtain the stochastic incompressible Euler equation via a singular limit argument.
Finally, we also mention the result \cite{breit2019stationary} in which stationary solutions are constructed for the compressible Navier--Stokes system driven by stochastic forces  in the full range of  adiabatic exponent $\gamma>\frac{3}{2}$ available for existence theory. 
\\
In the last years,  the quasineutral limit for the Navier--Stokes--Poisson system have also been widely studied in the deterministic setting in the case of weak and strong solutions. For instance, see  \cite{W04}  for the smooth solution setting with well-prepared initial data, or \cite{JLW08} in the case of weak solutions  both in the whole space and in the torus without restrictions on the viscous coefficients, with well-prepared initial data. A more general analysis in the context of weak solutions and in the framework of general initial data was performed  in \cite{DM12} where all the regularity and smallness assumptions of the previous paper were removed. In the contest of  combined quasineutral  and inviscid limit,  we mention  \cite{JW06} and \cite{DoFeNo15}. Finally  in  \cite{donatelli2008quasineutral},  the authors investigated, in the whole space,  the combined incompressible quasineutral limit of the isentropic Navier--Stokes--Poisson system   and obtained the convergence of weak solution of the Navier--Stokes--Poisson system to the weak solution of the incompressible Navier--Stokes equations by means of dispersive estimates of Strichartz's type. However, to our present knowledge,  there are no results  available for the combined low Mach and Debye length limit in the stochastic setting.
\\
Our aim in this paper  is to show in the framework of ill-prepared initial data that any such family of solutions parametrised by this Debye length, converges weakly in law, as the Debye length vanishes, to a solution of the stochastically-forced incompressible Navier--Stokes equation which is weak in the sense of probability and also weak in the sense of distributions. This result is virtually a stochastic analogue of \cite{donatelli2008quasineutral} with mainly, three significant improvements. Firstly, by introducing stochasticity in the model, we can account for turbulence in the physical system whereas from a purely mathematical point of view, we are presented with additional difficulties such as obtaining equivalent stochastic a priori estimates and compactness results. These improvements are especially non-trivial at the heart of our analyses which is to show that the gradient part (or the curl-free part) of our vector-valued dependent variables are not present in the limit system. Indeed,  in the case of ill-prepared initial data and incompressible limit analysis,  one of the main issue is the lost of compactness of the velocity field due to the presence of  acoustic waves. These are waves that oscillates at very high frequencies and are supported by the gradient part of the velocity.  Here, since we deal at the same time with the quasineutral limit, we also have  to consider the  high plasma oscillations, namely,   the presence of  stiff terms due to the electric field, whose oscillations could not  be in general  controlled only by the dispersion of the acoustic waves. Related to this last point is this second improvement of  \cite{donatelli2008quasineutral}.  In fact in \cite{donatelli2008quasineutral}, the analysis of the acoustic waves was performed by standard Strichartz estimates for the wave equation. In this paper, we obtain better dispersive estimates by utilising the \text{full} Klein--Gordon operator as opposed to just the  wave operator. Indeed, the  velocity field 
disperses and oscillates and the Klein--Gordon operator is able to capture this duality.  Moreover due to the stochastic setting  the acoustic equations  must be handled in a rather different manner, we have to  produce  new suitable dispersive estimate for first-order system of equations. As a result, we get that the dispersive behaviour dominates on the high frequency time oscillations, therefore the usual estimates of Strichartz-type are sufficient to pass to the limit in the velocity field and to control the electric field's  time oscillations (plasma oscillation). The third improvement of  \cite{donatelli2008quasineutral} is connected to this last part. Due to the difference in the analysis of the acoustic waves,  we obtain a better range of  admissible values of $\beta$ in \eqref{valuebeta} required for the convergence of the gradient part of the momentum and the velocity fields to zero and to perform the convergence analysis for the electric field. We will give further details of the above-mentioned techniques in the subsequent paragraphs where we describe the structure of this paper.
\\
In our preliminary section, Section \ref{sec:prelim}, we present some notations that will be used throughout this document and also state some analytic tools for deterministic PDEs. Since we have introduced a stochastic perturbation in \eqref{momEqB}, we also present in Section \ref{sec:prelim}, the assumptions on the noise coefficient which will ensure that we have a well-defined stochastic integral. Furthermore, since our ultimate goal is to explore the interplay between solutions of two sets of a system of equations, we also make precise, the notion of a solution to both systems. We finally end Section \ref{sec:prelim} with the statement of our main result, Theorem \ref{thm:wholespace}.
\\
After our preliminary section, the rest of our paper is devoted to the proof of our main result. We begin in Section \ref{sec:UniEst} by obtaining uniform a priori estimates via energy methods for the relevant sequence of functions. Among these will be the analysis of the momentum sequence which we decompose into its solenoidal part (or divergence-free part) and gradient part (or curl-free part) and treat these parts separately in Section \ref{sec:solePart} and Section \ref{sec:gradPart} respectively. The treatment of the gradient part of the momentum in Section \ref{sec:gradPart} will be the most important analysis in this paper. Here, our goal will be the derivation of dispersive estimates when we project the momentum equation \eqref{momEqB} onto weakly curl-free fields. This goal is achieved by tools from harmonic analysis and in particular, Strichartz-type estimates for Klein-Gordon equation; a second-order spacetime equation. 
Unfortunately, because we are dealing with a stochastic partial differential equation, the existing deterministic tools for treating this second-order equation breakdown. In particular, we can not replicate the analysis in \cite{donatelli2008quasineutral} in our setting since we can not make sense of a second-order SPDE. Therefore, we consider instead, a stochastic Klein-Gordon system of two first-order coupled equations, see \eqref{acousticSPD111LM}. By relying on the methods used in \cite{mensah2016existence, mensah2019theses} for the analysis of acoustic waves in the stochastic compressible Navier--Stokes system, we produce a similar dispersive estimate for our  first-order system of equations.   Since the treatment of the Klein-Gordon operator \eqref{KGoperator} is clearly more complicated than the wave operator treated in \cite{mensah2016existence, mensah2019theses}, obtaining dispersive estimate in our context is also more tricky. Nevertheless,  preliminary analyses into low and high frequencies regimes, see Proposition \ref{prop:expBound}, enables us to obtain our goal of showing that the gradient part of the momentum sequence vanishes in the limit as the Debye number goes to zero. We then end Section \ref{sec:UniEst} by showing in Section \ref{sec:gradVelo} that the corresponding gradient part of the velocity sequence also vanishes in the limit as the Debye number goes to zero.
\\
With all our crucial uniform estimates in hand, we proceed to show stochastic compactness of these uniformly bounded functions in Section \ref{subsec:compactness}. Since some of our functions live in spaces which are not Polish, our compactness arguments relies on the Jakubowski--Skorokhod representation  theorem \cite{jakubowski1998short} rather than the more conventional Prokhorov and Skorokhod theorems for functions in Polish spaces.
\\
To complete the proof of our main result, Theorem \ref{thm:wholespace}, we identify the incompressible stochastic Navier--Stokes equation as our limit system in Section \ref{sec:limitWholeSpace}. Here, we follow the original approach in \cite{breit2015incompressible, mensah2016existence} (and later polished in \cite{breit2018stoch}) to identify the limit system except for the treatment of the term due to the electric field. This is a purely deterministic argument and we thus follow the approach in \cite{donatelli2008quasineutral}. 
\\
A result  related to Theorem \ref{thm:wholespace} is the so called {\em zero electron mass number limit } result. Although being a different scaled Navier--Stokes--Poisson system in the strict physical sense of the word, formally speaking, this corresponds to setting $\beta=0$ in \eqref{valuebeta}--\eqref{elecFieldB} above and studying the limit as $\varepsilon \rightarrow 0$. We, therefore, present in the  Section \ref{appendix}, the exact scaled model and make rigorous, the formal notion of getting a limiting system when $\beta$ in \eqref{valuebeta}--\eqref{elecFieldB} is set to zero. This is given in Theorem \ref{thm:electronMass} and we obtain the same limit system (i.e., the stochastic incompressible Navier--Stokes equation) as we do for our combined incompressible quasineutral limit result, Theorem  \ref{thm:wholespace}. Indeed, the proof of this zero electron mass number limit result reduces to a corollary of the combined incompressible quasineutral limit result when one replaces the analysis of the \textit{singular} homogeneous Klein--Gordon equation \eqref{homogeAcoustic} and its solution \eqref{solution} with the analysis of its  \textit{non-singular} version as presented in Proposition \ref{nonsingularKG}. Setting $\beta=0$ everywhere else in the proof of Theorem \ref{thm:wholespace}, completes the proof of  Theorem \ref{thm:electronMass}.

\section{Preliminaries and Main Results}
\label{sec:prelim}
\subsection{Notations} The following notations will the used throughout this work.
\begin{enumerate}
\item For functions $F$ and $G$ and a variable $p$, we denote by $F \lesssim G$ and $F \lesssim_p G$, the existence of  a generic constant $c>0$ and another such constant $c(p)>0$ which now depends on $p$ such that $F \leq c\,G$ and $F \leq c(p) G$ respectively.
\item We write $K\Subset \mathbb{R}^3$ if $K \subseteq \mathbb{R}^3$ and $K$ is compact. 
\item Let $\Delta^{-1}$ be the inverse of the Laplacian on $\mathbb{R}^3$ and  we denote by $\mathcal{Q}:= \nabla \Delta^{-1} \mathrm{div}$ and $\mathcal{P}:= \mathbb{I} - \mathcal{Q}$, Helmoltz decomposition onto gradient and solenoidal vector fields respectively.
\item The homogeneous Sobolev space with differentiability $m\in \mathbb{R}$ and  integrability $q\geq1$ is denoted by $D^{m,q}(\mathbb{R}^3)$, see \cite{galdi2011introduction}, whereas
\begin{align*}
L^q_2(\mathbb{R}^3)=\big\{f\in L^1_{\mathrm{loc}}(\mathbb{R}^3) \, :\, \vert f \vert \chi_{\{2\vert f\vert\leq 1 \}}\in L^2(\mathbb{R}^3),\quad
\vert f \vert \chi_{\{2\vert f\vert> 1 \}}\in L^q(\mathbb{R}^3)
\big\}
\end{align*} 
is the Orlicz space of integrability $q\geq1$. The usual Lebesgue space of vectors which are weakly solenoidal is denoted by $L^q_{\mathrm{div}}(\mathbb{R}^3)$ with a similar notation for Sobolev spaces $W^{s,q}_{\mathrm{div}}(\mathbb{R}^3)$, $s\in \mathbb{R}$,  $D^{m,q}_{\mathrm{div}}(\mathbb{R}^3)$ and compactly supported smooth functions $C^\infty_{c,\mathrm{div}}(\mathbb{R}^3)$.
\end{enumerate}

\subsection{Some analytic results}
We collect in this section, some classical analytic results which are required in the sequel. The following lemma gives some estimates for mollified functions, see \cite{desjardins1999low, donatelli2008quasineutral}.
\begin{lemma}
For  $\kappa \in (0,1)$, let $\wp_\kappa$ be the usual mollifier. Then
\begin{enumerate}
\item For any $f \in D^{1,2}(\mathbb{R}^d)$, we have
\begin{align}
\label{mollifierA}
\Vert f - [f]_\kappa \Vert_{L^p(\mathbb{R}^d)}
\lesssim_p \kappa^{1-d\left( \frac{1}{2} -\frac{1}{p}\right)}
\Vert \nabla f \Vert_{L^2(\mathbb{R}^d)}
\end{align}
where $[f]_\kappa := f * \wp_\kappa$, $p\in [2, \infty)$ if $d=2$ and $p\in [2,6]$ if $d=3$.
\item Furthermore,
\begin{align}
\label{mollifierB}
\Vert [f]_\kappa \Vert_{L^p(\mathbb{R}^d)}
\lesssim_p \kappa^{-s-d\left( \frac{1}{r} -\frac{1}{p}\right)}
\Vert  f \Vert_{W^{-s,r}(\mathbb{R}^d)}
\end{align}
for any $p,r\in [1, \infty]$, $r\leq p$ and $s\geq 0$.
\end{enumerate}
\end{lemma}

\subsection{Assumptions on the stochastic force}
\label{sec:noiseAssumption}
Let $\big(\Omega, \mathscr{F}, (\mathscr{F}_t)_{t\geq 0}, \mathbb{P} \big)$ be a stochastic basis endowed with  a right-continuous filtration $(\mathscr{F}_t)_{t\geq 0}$. Let $W(t)$ be an $(\mathscr{F}_t)$-cylindrical Wiener process satisfying $W(t) = \sum_{k\in \mathbb{N}} \beta_k(t) e_k$ for $t\in [0,T]$
where $(\beta_k)_{k\in \mathbb{N}}$ is a family of mutually independent Brownian motions and $(e_k)_{k\in \mathbb{N}}$ are orthonormal basis of a separable Hilbert space $\mathfrak{U}$. Now consider the larger space $\mathfrak{U}_0 \supset  \mathfrak{U}$ for which the  embedding $\mathfrak{U}\hookrightarrow \mathfrak{U}_0$ is Hilbert--Schmidt and for which $W$ has $\mathbb{P}$-a.s. $C([0,T];\mathfrak{U}_0)$ sample paths. To ensure that the stochastic integral $\int_0^\cdot\mathbf{G}(\varrho,\varrho \mathbf{u})\mathrm{d}W$ is a well-defined $(\mathscr{F}_t)$-martingale taking value in  a suitable Hilbert space, we set $\mathbf{m}=\varrho \mathbf{u}$ where $0\leq \varrho \in L^\gamma(\mathbb{R}^3)$ and $\mathbf{u} \in L^2(\mathbb{R}^3)$ and let $\mathbf{G}(\varrho,  \mathbf{m}):\mathfrak{U}\rightarrow    L^1(  K  )$ for $K\Subset \mathbb{R}^3$ be defined as follows:
\\
Assume there exist $C^1_c$-functions  $ \mathbf{g}_k: \mathbb{R}^3\times \mathbb{R}_{\geq0}  \times \mathbb{R}^3  \rightarrow \mathbb{R}^3$   such that
\begin{align}
&\mathbf{G} (\varrho,  \mathbf{m}) e_k   =   \mathbf{g}_k(\cdot, \varrho(\cdot),  \mathbf{m} (\cdot)), \label{noiseAssumptionWholespace1}
\\
&\sum_{k\in \mathbb{N}}\vert  \mathbf{g}_k(x, \varrho,  \mathbf{m})   \vert^2  \lesssim  \varrho^2 + \vert \mathbf{m}\vert^2 ,
\quad
\sum_{k\in \mathbb{N}}\big\vert\nabla_{\varrho,\mathbf{m}} \, \mathbf{g}_k(x, \varrho, \mathbf{m}) \big) \big\vert^2   \lesssim 1. \label{noiseAssumptionWholespace2}
\end{align}
\subsection{Concepts of solution}
\label{sec:solution}
We now define the notions of  solution that we wish to consider. A reader may refer to \cite{breit2018stoch} for any terminology that they are not familiar with.
\begin{definition}
\label{def:martSolutionExis}
Let $\Lambda$ be a Borel probability measure on $L^1(\mathbb{R}^3)\times L^1(\mathbb{R}^3)$. For $\varepsilon>0$ fixed, we say that 
$\big[(\Omega^\varepsilon ,\mathscr{F}^\varepsilon ,(\mathscr{F}^\varepsilon _t)_{t\geq0},\mathbb{P}^\varepsilon );\varrho^\varepsilon , \mathbf{u}^\varepsilon , V^\varepsilon , W^\varepsilon   \big]$
is a \textit{finite energy weak martingale solution} of \eqref{contEqB}--\eqref{elecFieldB} with initial law $\Lambda$ provided:
\begin{enumerate}
\item $(\Omega^\varepsilon,\mathscr{F}^\varepsilon,(\mathscr{F}^\varepsilon_t),\mathbb{P}^\varepsilon)$ is a stochastic basis with a complete right-continuous filtration and $W^\varepsilon$ is a $(\mathscr{F}^\varepsilon_t)$-cylindrical Wiener process;
\item the triplets $(\varrho^\varepsilon,\mathbf{u}^\varepsilon,V^\varepsilon) $ are  $(\mathscr{F}^\varepsilon_t)$-adapted random distributions;
\item there exists $\mathscr{F}^\varepsilon_0$-measurable random variables $(\varrho^\varepsilon_0, \varrho^\varepsilon_0\mathbf{u}^\varepsilon_0)$ such that $\Lambda = \mathbb{P}\circ (\varrho^\varepsilon_0, \varrho^\varepsilon_0 \mathbf{u}^\varepsilon_0)^{-1}$;
\item for all $\psi \in C^\infty_c (\mathbb{R}^3)$ and $\bm{\phi} \in C^\infty_c (\mathbb{R}^3)$ and all $t\in [0,T]$, the following
\begin{equation}
\begin{aligned}
\label{eq:distributionalSol}
 \int_{\mathbb{R}^3}&\varrho^\varepsilon  \psi \,\mathrm{d}x   =   \int_{\mathbb{R}^3} \varrho^\varepsilon _0 \psi \mathrm{d}x  +  \int_0^t  \int_{\mathbb{R}^3} \varrho^\varepsilon  \mathbf{u}^\varepsilon  \cdot\nabla \psi \,\mathrm{d}x  \mathrm{d}s,
\\
 \int_{\mathbb{R}^3}&\varrho^\varepsilon  \mathbf{u}^\varepsilon \cdot \bm{\phi} \,\mathrm{d}x   =   \int_{\mathbb{R}^3} \varrho ^\varepsilon_0 \mathbf{u}^\varepsilon _0 \cdot \bm{\phi} \,\mathrm{d}x  +  \int_0^t  \int_{\mathbb{R}^3} \varrho^\varepsilon  \mathbf{u}^\varepsilon \otimes \mathbf{u}^\varepsilon : \nabla \bm{\phi}\,\mathrm{d}x  \mathrm{d}s 
\\&- {\nu}_1\int_0^t  \int_{\mathbb{R}^3} \nabla \mathbf{u}^\varepsilon : \nabla \bm{\phi} \,\mathrm{d}x  \mathrm{d}s 
-
({\nu}_2+{\nu}_1 )\int_0^t  \int_{\mathbb{R}^3} \mathrm{div}\,\mathbf{u}^\varepsilon \, \mathrm{div}\, \bm{\phi}\,\mathrm{d}x  \mathrm{d}s    
\\&+  \frac{1}{\varepsilon^2}\int_0^t  \int_{\mathbb{R}^3} (\varrho^\varepsilon )^\gamma \mathrm{div} \, \bm{\phi} \,\mathrm{d}x  \mathrm{d}s
+
\frac{1}{\varepsilon^2}\int_0^t  \int_{\mathbb{R}^3} \varrho^\varepsilon  \nabla V^\varepsilon  
\cdot
\bm{\phi}\,\mathrm{d}x    \mathrm{d}s
\\&+
\int_0^t  \int_{\mathbb{R}^3} \mathbf{G}(\varrho^\varepsilon ,  \mathbf{m}^\varepsilon ) \cdot \bm{\phi} \,\mathrm{d}x\, \mathrm{d}W^\varepsilon 
\end{aligned}
\end{equation}
hold $\mathbb{P}^\varepsilon$-a.s.
\item the  energy inequality 
\begin{equation}
\begin{aligned}
\label{augmentedEnergy}
\int_{\mathbb{R}^3}\bigg[ &\frac{1}{2} \varrho^\varepsilon  \vert \mathbf{u}^\varepsilon  \vert^2  
+\frac{1}{\varepsilon^2}H(\varrho^\varepsilon,  1) + \varepsilon^{\beta-2} \big\vert \nabla V^\varepsilon  \big\vert^2
\bigg](t)\,\mathrm{d}x
+{\nu}_1
\int_0^t \int_{\mathbb{R}^3}\vert  \nabla \mathbf{u}^\varepsilon  \vert^2 \,\mathrm{d}x  \,\mathrm{d}s
\\
&+({\nu}_2+{\nu}_1)
\int_0^t \int_{\mathbb{R}^3}\vert  \mathrm{div}\, \mathbf{u}^\varepsilon  \vert^2 \,\mathrm{d}x  \,\mathrm{d}s
\leq
 \int_{\mathbb{R}^3}\bigg[\frac{1}{2} \varrho^\varepsilon _0\vert \mathbf{u} ^\varepsilon_0 \vert^2  
+\frac{1}{\varepsilon^2} H(\varrho^\varepsilon _0,   1) + \varepsilon^{\beta-2} \big\vert \nabla V^\varepsilon_0  \big\vert^2 \bigg]\,\mathrm{d}x
\\&
+\frac{1}{2}\int_0^t
 \int_{\mathbb{R}^3}
 (\varrho^\varepsilon)^{-1}
 \sum_{k\in\mathbb{N}} \big\vert
 \mathbf{g}_k(x,\varrho^\varepsilon ,\mathbf{m} ^\varepsilon ) 
\big\vert^2\,\mathrm{d}x\,\mathrm{d}s
+\int_0^t  \int_{\mathbb{R}^3}  \mathbf{u}^\varepsilon \cdot\mathbf{G}(\varrho^\varepsilon , \mathbf{m}^\varepsilon )\,\mathrm{d}x\,\mathrm{d}W^\varepsilon ;
\end{aligned}
\end{equation}
holds $\mathbb{P}$-a.s. for a.e. $t\in[0,T]$ where
\begin{equation} 
\begin{aligned}
\label{HVarrhoOverline}
H(\varrho^\varepsilon , 1) 
= \frac{1}{(\gamma -1)}\big[(\varrho^\varepsilon) ^\gamma    -  \gamma (\varrho^\varepsilon  - 1)   - 1 \big].
\end{aligned}
\end{equation}
\end{enumerate}
\end{definition} 
\begin{remark}
At this point, we draw the reader's attention to the differences in the noise terms in \eqref{eq:distributionalSol} and \eqref{augmentedEnergy} as opposed to their counterpart in \cite{donatelli2020dissipative}. This is because we have chosen a more general noise in \eqref{momEqB} compared to the version in \cite{donatelli2020dissipative} which is a multiple of the fluid's density. Beside  the current version being more general, it also enjoys better estimates especially when the density has low regularity. 
%As an exercise, the reader may try to replicate estimates for noise terms in the sequel  when they are replaced by the version in \cite{donatelli2020dissipative}. 
\end{remark}
\subsection{The limit system}
We now give the precise definitions of a solution to the anticipated limit system
\begin{equation}
\begin{aligned}
\label{incomprSPDE}
&\mathrm{div}( \mathbf{U}) = 0, 
\\
&\mathrm{d}( \mathbf{U}) + \big[\mathrm{div}( \mathbf{U}\otimes \mathbf{U}) - \nu_1\Delta \mathbf{U}  +\nabla \pi \big]\mathrm{d}t = \mathbf{G}(1, \mathbf{U})\mathrm{d}W
\end{aligned}
\end{equation}
where $\pi$ is the limit pressure.
\begin{definition}
\label{def:martSolutionIncompre}
If $\Lambda$ \ \ is a Borel probability measure on \ \ $L^2_{\mathrm{div}}(\mathbb{R}^3)$, then
we say that \\
$[(\Omega,\mathscr{F},(\mathscr{F}_t),\mathbb{P}), \mathbf{U}, W ]$ is a \textit{weak martingale solution} of \eqref{incomprSPDE} with initial law $\Lambda$ provided:
\begin{enumerate}
\item $(\Omega,\mathscr{F},(\mathscr{F}_t),\mathbb{P})$ is a stochastic basis with a complete right-continuous filtration and  $W$ is a $(\mathscr{F}_t)$-cylindrical Wiener process;
\item $\mathbf{U}$ is $(\mathscr{F}_t)$-adapted and $\mathbf{U}\in   C_w\left( [0,T];L^2_{\mathrm{div}}(\mathbb{R}^3) \right) \cap L^2(0,T;W^{1,2}_{\mathrm{div}}(\mathbb{R}^3)  )$ $ \mathbb{P}$-a.s.;
\item there exists $\mathscr{F}_0$-measurable random variable $\mathbf{U}_0$ such that $\Lambda = \mathbb{P}\circ (\mathbf{U}_0)^{-1}$,
\item for all $\bm{\phi} \in C^\infty_{c,\mathrm{div}} (\mathbb{R}^3)$  and for all $t\in [0,T]$, the following
\begin{align*}
\int_{\mathbb{R}^3}  \mathbf{U}(t)\cdot \bm{\phi} \, \mathrm{d}x
= 
\int_{\mathbb{R}^3}  \mathbf{U}(0)\cdot \bm{\phi}\, \mathrm{d}x\, \mathrm{d}s   &+  \int_0^t \int_{\mathbb{R}^3} \big[  \mathbf{U}\otimes \mathbf{U}   - \nu_1 \nabla \mathbf{U} 
\big]:\nabla \bm{\phi}\, \mathrm{d}x\, \mathrm{d}s 
\\&+
\int_0^t \int_{\mathbb{R}^3} \mathbf{G}(1, \mathbf{U}) \cdot \bm{\phi} \, \mathrm{d}x\, \mathrm{d}W(s)
\end{align*}
holds $\mathbb{P}$-a.s.
%\item  the  energy inequality 
%\begin{equation}
%\begin{aligned}
%\int_{\mathbb{R}^3} &\frac{1}{2}   \vert \mathbf{U}  \vert^2  
%(t)\,\mathrm{d}x
%+\overline{\nu}_1
%\int_0^t \int_{\mathbb{R}^3}\vert  \nabla \mathbf{U}  \vert^2 \,\mathrm{d}x  \,\mathrm{d}s
%\leq
% \int_{\mathbb{R}^3}\frac{1}{2} \vert \mathbf{U} _0 \vert^2  
%\,\mathrm{d}x
%\\&
%+\frac{1}{2}\int_0^t
% \int_{\mathbb{R}^3}
% \sum_{k\in\mathbb{N}} \big\vert
% \mathbf{g}_k(x,1 ,\mathbf{U}  ) 
%\big\vert^2\,\mathrm{d}x\,\mathrm{d}s
%+\int_0^t  \int_{\mathbb{R}^3}   \mathbf{U} \cdot\mathbf{G}(1 , \mathbf{U} )\,\mathrm{d}x\,\mathrm{d}W ;
%\end{aligned}
%\end{equation}
%holds $\mathbb{P}$-a.s. for a.e. $t\in[0,T]$
\end{enumerate}
\end{definition}
The existence of weak martingale solution of \eqref{incomprSPDE} in the sense of Definition \ref{def:martSolutionIncompre} above  has been shown  in \cite{capinski1994stochastic, flandoli1995martingale} for bounded domain. For the case of the whole space, refer to \cite{mikulevicius2005global, mensah2016existence}.
\subsection{Main Results}
We now state our main result. It establishes the convergence, as $\varepsilon\to 0$, of any family of solutions to  \eqref{contEqB}--\eqref{elecFieldB} in the sense of Definition \ref{def:martSolutionExis} to a solution of \eqref{incomprSPDE}  in the sense of  Definition \ref{def:martSolutionIncompre}.
\begin{theorem}
\label{thm:wholespace}
Let  $\Lambda$ be a given Borel probability measure on $L^2_{\mathrm{div}}(\mathbb{R}^3)$. For $\varepsilon >0$ and  $\gamma>3/2$, we let $\Lambda^\varepsilon$ be a family of Borel probability measures on $\big[ L^1_x \big]^2= L^1(\mathbb{R}^3) \times L^1(\mathbb{R}^3)$ such that
\begin{equation}
\begin{aligned}
\label{intialLawepsilon}
\Lambda^\varepsilon  \Big\{ (\varrho, \mathbf{m})  \, : \,   
 \vert \varrho-1 \vert\leq  \varepsilon M  \Big\}=1
\end{aligned}
\end{equation}
holds for a deterministic constant $M>0$ which is independent of $\varepsilon>0$. For all $p\in[1,\infty)$, we assume that the following moment estimate
\begin{equation}
\begin{aligned}
\label{initialLaw}
\int_{[ L^1_x]^2}  \left\Vert \frac{1}{2}\frac{\vert\mathbf{m}\vert^2}{\varrho}  +  \frac{1}{\varepsilon^2} H(\varrho,1) 
+
\varepsilon^{\beta-2} \vert \nabla V \vert^2 \right\Vert^p_{L^1_x}   \mathrm{d}\Lambda^\varepsilon(\varrho,  \mathbf{m}) \lesssim_{p} 1,
\end{aligned}
\end{equation}
holds uniformly in $\varepsilon$.
Further assume that \eqref{noiseAssumptionWholespace1}-- \eqref{noiseAssumptionWholespace2} holds and that  the marginal law of $\Lambda^\varepsilon$ corresponding to the second component converges to $\Lambda$ weakly in the sense of measures on $L^{\frac{2\gamma}{\gamma+1}}(\mathbb{R}^3 )$. If
\begin{align}
\label{familyweakMartSol}
\left[(\Omega^\varepsilon,\mathscr{F}^\varepsilon,(\mathscr{F}^\varepsilon_t),\mathbb{P}^\varepsilon);\varrho^\varepsilon, \mathbf{u}^\varepsilon, V^\varepsilon, W^\varepsilon  \right]
\end{align}
is a finite energy weak martingale solution of \eqref{contEqB}--\eqref{elecFieldB} in the sense of Definition \ref{def:martSolutionExis} with initial law $\Lambda^\varepsilon$
and for any $\delta>0$,
\begin{align}
0<\beta <\frac{1}{2+\delta},
\end{align}
then
\begin{align*}
\varrho^\varepsilon \rightarrow 1\quad&\text{in law in}\quad L^\infty(0,T;L^{\gamma}_{\mathrm{loc}} (\mathbb{R}^3 )),
\\
 \mathbf{u}^\varepsilon \rightarrow  \mathbf{U}\quad& \text{in law in}\quad 
  \big(L^2(0,T;W^{1,2}(\mathbb{R}^3 )),w \big)
\end{align*}
as $\varepsilon \rightarrow0$ and  $\mathbf{U}$ is a weak martingale solution of \eqref{incomprSPDE}  in the sense of  Definition \ref{def:martSolutionIncompre} with  initial law $\Lambda$.
\end{theorem}
\begin{remark}
Notice that  the conditions \eqref{intialLawepsilon} and \eqref{initialLaw} correspond to the ill prepared initial data setting. These are the minimum requirements that we can impose on initial data so that the solutions of the Navier--Stokes--Poisson system are uniformly bounded as $\varepsilon\to 0$.
\end{remark}
\section{Uniform estimates}
\label{sec:UniEst}
%************** *DON'T DELETE THIS COMMENT ********************
%Note that $\varrho \mapsto H(\varrho,1)$ is convex with minimum $(\varrho, H) =(1,0)$ and so $H\geq 0$. Therefore there exist $c>0$ such that $\varrho \lesssim 1+H(\varrho,1) \lesssim 1+\frac{1}{\varepsilon}H(\varrho,1)$ since $\varepsilon>0$ is small. An example is the following. if $0\leq H$ then
%\begin{align*}
%-\varrho+1\leq H-\varrho +1 \leq 2[H-\varrho +1]
%\end{align*}
%and so
%\begin{align*}
%\varrho\leq 2H +1 \leq 2(H+1). \leq 2(\varepsilon^{-1} H+1)
%\end{align*}
%
%****************************
Without loss of generality, 
see \cite{jakubowski1998short}, it suffices to consider the following family
\begin{align}
\big[(\Omega ,\mathscr{F} ,(\mathscr{F} _t)_{t\geq0},\mathbb{P} );\varrho^\varepsilon , \mathbf{u}^\varepsilon , V^\varepsilon , W   \big]
\end{align}
instead of \eqref{familyweakMartSol}, see also \cite[Remark 4.0.4]{breit2018stoch} and  \cite[Remark 6.2.15]{mensah2019theses}.
So by \eqref{augmentedEnergy}, we have that
\begin{equation}
\begin{aligned}
\label{energy1}
\int_{\mathbb{R}^3}\bigg[ &\frac{1}{2} \varrho^\varepsilon  \vert \mathbf{u}^\varepsilon  \vert^2  
+   \frac{1}{\varepsilon^2} H(\varrho^\varepsilon ,  1) + \varepsilon^{\beta-2} \big\vert \nabla V^\varepsilon  \big\vert^2
\bigg](t)\,\mathrm{d}x
+\nu_1
\int_0^t \int_{\mathbb{R}^3}\vert  \nabla \mathbf{u}^\varepsilon  \vert^2 \,\mathrm{d}x  \,\mathrm{d}s
\\
&+(\nu_2+\nu_1)
\int_0^t \int_{\mathbb{R}^3}\vert  \mathrm{div}\, \mathbf{u}^\varepsilon  \vert^2 \,\mathrm{d}x  \,\mathrm{d}s
\leq
 \int_{\mathbb{R}^3}\bigg[\frac{1}{2} \varrho^\varepsilon _0\vert \mathbf{u}^\varepsilon _0 \vert^2  
+  \frac{1}{\varepsilon^2} H(\varrho^\varepsilon _0,   1) + \varepsilon^{\beta-2} \big\vert \nabla V^\varepsilon_0  \big\vert^2 \bigg]\,\mathrm{d}x
\\&
+\frac{1}{2}\int_0^t
 \int_{\mathbb{R}^3}
(\varrho^\varepsilon)^{-1}
 \sum_{k\in\mathbb{N}} \big\vert
 \mathbf{g}_k(x,\varrho^\varepsilon ,\mathbf{m}^\varepsilon  ) 
\big\vert^2\,\mathrm{d}x\,\mathrm{d}s
+\int_0^t  \int_{\mathbb{R}^3}   \mathbf{u}^\varepsilon \cdot\mathbf{G}(\varrho^\varepsilon , \mathbf{m}^\varepsilon )\,\mathrm{d}x\,\mathrm{d}W ;
\end{aligned}
\end{equation}
holds $\mathbb{P}$-a.s. for a.e. $t\in[0,T]$. By \eqref{noiseAssumptionWholespace2}, we obtain
\begin{equation}
\begin{aligned}
\label{seqNois1}
\mathbb{E} &\sup_{t\in [0,T]} \bigg\vert \int_0^t
 \int_{\mathbb{R}^3}
\frac{1}{\varrho^\varepsilon}
 \sum_{k\in\mathbb{N}} \big\vert
 \mathbf{g}_k(x,\varrho^\varepsilon ,\mathbf{m}^\varepsilon  ) 
\big\vert^2\mathrm{d}x\mathrm{d}s \bigg\vert^p
\lesssim
\mathbb{E}\bigg[ \int_0^T
 \int_{K}
\frac{1}{\varrho^\varepsilon}
 \sum_{k\in\mathbb{N}} \big\vert
 \mathbf{g}_k(x,\varrho^\varepsilon ,\mathbf{m}^\varepsilon  ) 
\big\vert^2\mathrm{d}x\mathrm{d}s \bigg]^p
\\&
\lesssim
\mathbb{E}\bigg[ \int_0^T \int_{K}\frac{1}{\varrho^\varepsilon}\bigg(\frac{1}{2} \vert\varrho^\varepsilon \mathbf{u}^\varepsilon \vert^2
+
 (\varrho^\varepsilon)^2  \bigg) \, \mathrm{d}x \, \mathrm{d}t \bigg]^p
 \lesssim
\mathbb{E}\bigg[ \int_0^T \int_{K} \bigg(\frac{1}{2}\varrho^\varepsilon \vert \mathbf{u}^\varepsilon \vert^2
+
 \varrho^\varepsilon  \bigg) \, \mathrm{d}x \, \mathrm{d}t \bigg]^p
\\&\lesssim
1+
\mathbb{E}\bigg[ \int_0^T \int_{\mathbb{R}^3} \bigg( \frac{1}{2}\varrho^\varepsilon \vert \mathbf{u}^\varepsilon \vert^2+  \frac{1}{\varepsilon^2}
H(\varrho^\varepsilon,1) \bigg) \, \mathrm{d}x \, \mathrm{d}t \bigg]^p
\end{aligned}
\end{equation}
where $K\Subset \mathbb{R}^3$ is the support of $\mathbf{G}$, recall Section \ref{sec:noiseAssumption}. Note that we have also used the following uniform-in-$\varepsilon$ inequality $\varrho^\varepsilon \lesssim 1 + \frac{1}{\varepsilon^2}
H(\varrho^\varepsilon,1)$ in the above.
Also, it follows from the Burkholder--Davis--Gundy's inequality, Young's inequality and a similar estimate as \eqref{seqNois1} that
\begin{equation}
\begin{aligned}
\label{seqNois2}
\mathbb{E} &\bigg( \sup_{t\in [0,T]} \bigg\vert \int_0^t  \int_{\mathbb{R}^3}  \mathbf{u}^\varepsilon \cdot\mathbf{G}(\varrho^\varepsilon , \mathbf{m}^\varepsilon )\,\mathrm{d}x\,\mathrm{d}W \bigg\vert \bigg)^p
\\&
\lesssim
\mathbb{E} \bigg[\int_0^T \sum_{k\in \mathbb{N}}
\bigg(   \int_{\mathbb{R}^3}  \sqrt{\varrho^\varepsilon} \mathbf{u}^\varepsilon \cdot \frac{1}{\sqrt{\varrho^\varepsilon}}\mathbf{g}_k(x,\varrho^\varepsilon , \mathbf{m}^\varepsilon )\,\mathrm{d}x \bigg)^2\, \mathrm{d}t \bigg]^\frac{p}{2}
\\&\lesssim
\mathbb{E} \bigg[\int_0^T 
\bigg(   \int_{\mathbb{R}^3}  \frac{1}{2} \varrho^\varepsilon \vert \mathbf{u}^\varepsilon \vert^2 \bigg)  \bigg( \int_{\mathbb{R}^3} \frac{1}{\varrho^\varepsilon} \sum_{k\in \mathbb{N}} \vert \mathbf{g}_k(x,\varrho^\varepsilon , \mathbf{m}^\varepsilon )\vert^2 \,\mathrm{d}x \bigg)\, \mathrm{d}t \bigg]^\frac{p}{2}
\\&\lesssim
1+
\delta\,
\mathbb{E}\bigg[ \sup_{t\in[0,T]} \int_{\mathbb{R}^3}  \frac{1}{2}\varrho^\varepsilon \vert \mathbf{u}^\varepsilon \vert^2\, \mathrm{d}x  \bigg]^p
+
C(\delta)\mathbb{E}\bigg[ \int_0^T \int_{\mathbb{R}^3} \bigg( \frac{1}{2}\varrho^\varepsilon \vert \mathbf{u}^\varepsilon \vert^2+ \frac{1}{\varepsilon^2}
H(\varrho^\varepsilon,1) \bigg) \mathrm{d}x \mathrm{d}t \bigg]^p
\end{aligned}
\end{equation}
for any $\delta>0$.
By taking the $p$-th moment in \eqref{energy1}, using \eqref{seqNois1}--\eqref{seqNois2} and applying  Gronwall's lemma, we obtain
\begin{equation}
\begin{aligned}
\label{energy2}
\mathbb{E}&\bigg[ \sup_{t\in [0,T]} \int_{\mathbb{R}^3} \frac{1}{2} \varrho^\varepsilon  \vert \mathbf{u}^\varepsilon  \vert^2  \,\mathrm{d}x  \bigg]^p
+
\mathbb{E} \bigg[ \sup_{t\in [0,T]} \int_{\mathbb{R}^3} \frac{1}{\varepsilon^2} H(\varrho^\varepsilon ,  1) \,\mathrm{d}x  \bigg]^p
\\&
+
\mathbb{E} \bigg[ \sup_{t\in [0,T]} \int_{\mathbb{R}^3} \varepsilon^{\beta-2} \big\vert \nabla V^\varepsilon  \big\vert^2
\,\mathrm{d}x \bigg]^p
+\nu_1
\mathbb{E} \bigg[ 
\int_0^T \int_{\mathbb{R}^3}\vert  \nabla \mathbf{u}^\varepsilon  \vert^2 \,\mathrm{d}x  \,\mathrm{d}t \bigg]^p
\\&+(\nu_2+\nu_1)\mathbb{E} \bigg[ 
\int_0^T \int_{\mathbb{R}^3}\vert  \mathrm{div}\, \mathbf{u}^\varepsilon  \vert^2 \,\mathrm{d}x  \,\mathrm{d}t \bigg]^p
\\&
\lesssim 1+
\mathbb{E} \bigg(
 \int_{\mathbb{R}^3}\bigg[\frac{1}{2} \varrho^\varepsilon _0\vert \mathbf{u}^\varepsilon _0 \vert^2  
+ \frac{1}{\varepsilon^2} H(\varrho^\varepsilon _0,  1) + \varepsilon^{\beta- 2} \big\vert \nabla V^\varepsilon_0  \big\vert^2 \bigg]\,\mathrm{d}x \bigg)^p
\end{aligned}
\end{equation}
uniformly of $\varepsilon>0$. Now since
\begin{equation}
\begin{aligned}
&\mathbb{E} \bigg(
 \int_{\mathbb{R}^3}\bigg[\frac{1}{2} \varrho^\varepsilon _0\vert \mathbf{u}^\varepsilon _0 \vert^2  
+ \frac{1}{\varepsilon^2} H(\varrho^\varepsilon _0,  1) + \varepsilon^{\beta-2} \big\vert \nabla V^\varepsilon_0  \big\vert^2 \bigg]\,\mathrm{d}x \bigg)^p
\\&=
\int_{[L^1_x]^2} \bigg\vert \int_{\mathbb{R}^3}\bigg[\frac{1}{2}\frac{\vert \mathbf{m}  \vert^2}{\varrho}  
+\frac{1}{\varepsilon^2} H(\varrho, 1) + \varepsilon^{\beta-2} \big\vert \nabla V  \big\vert^2 \bigg]\,\mathrm{d}x\bigg\vert^p \, \mathrm{d}\Lambda^\varepsilon(\varrho, \mathbf{m}) 
\lesssim_p 1,
\end{aligned}
\end{equation}
holds uniformly of $\varepsilon>0$ by \eqref{initialLaw}, it follows from \eqref{energy2} that
\begin{equation}
\begin{aligned}
\label{energyEst1}
&\mathbb{E}\bigg[ \sup_{t\in[0,T]}\bigg\Vert \frac{1}{\varepsilon^2} H(\varrho^\varepsilon, 1 ) \bigg\Vert_{L^1(\mathbb{R}^3)} \bigg]^p \, 
\lesssim_p 1,
&
\mathbb{E} \bigg[ \sup_{t\in[0,T]}\bigg\Vert \frac{1}{2}\varrho^\varepsilon\vert  \mathbf{u}^\varepsilon \vert^2 \bigg\Vert_{L^1(\mathbb{R}^3)} \bigg]^p \, 
\lesssim_p 1,
\\
&\mathbb{E} \bigg[ \sup_{t\in[0,T]}\Big\Vert \varepsilon^{\beta-2} \big\vert \nabla V^\varepsilon  \big\vert^2 \Big\Vert_{L^1(\mathbb{R}^3)} \bigg]^p \, 
\lesssim_p 1,
&
\mathbb{E}\bigg[\int_0^T \Big\Vert \vert \nabla\mathbf{u}^\varepsilon \vert^2 \Big\Vert_{L^1(\mathbb{R}^3)}   \,\mathrm{d}t\,\bigg]^p \, 
\lesssim_p 1
\end{aligned}
\end{equation}
uniformly in $\varepsilon$ for any $p\in [1,\infty)$. 
We also obtain from the last estimate in  \eqref{energyEst1} combined with Sobolev's embedding that the estimate
\begin{align}
\label{velocityl6}
\mathbb{E}\bigg[\int_0^T \Vert \mathbf{u}^\varepsilon \Vert_{L^6(\mathbb{R}^3)}^2   \,\mathrm{d}t\,\bigg]^p \, 
\lesssim_p 1
\end{align}
holds uniformly in $\varepsilon$. Furthermore, we have the following result.
\begin{lemma}
Let  $\overline{\gamma}=\min \{\gamma,2 \}$. The following estimate
\begin{equation}
\begin{aligned}
\label{denFlucEst}
\mathbb{E} \bigg[  \sup_{t\in[0,T]}\Vert \sigma^\varepsilon \Vert_{L^{\overline{\gamma}}_{2}(\mathbb{R}^3)} \bigg]^p \lesssim_p 1
\end{aligned}
\end{equation}
holds uniformly in $\varepsilon>0$ where
\begin{align}
\label{densityFluctuationFunction}
\sigma^\varepsilon = \frac{\varrho^\varepsilon-1}{\varepsilon}
\end{align}
is the density fluctuation. 
\end{lemma}
\begin{proof}
The result follows from the deterministic case which is shown in for example \cite[Eqn. (23)]{donatelli2008quasineutral}, see also \cite{lions1998incompressible}. However, for further clarification and for the use in subsequent computation below, we present the proof here.\\
First of all, just as in \cite[Eqn. (23)]{donatelli2008quasineutral},
by using the convexity of the mapping $z \mapsto z^\gamma - 1  -\gamma(z-1)$ and the first estimate in \eqref{energyEst1}, if $\gamma<2$ (so that $\varepsilon^2< \varepsilon^\gamma$ for $\varepsilon>0$ small) we can conclude that
\begin{equation}
\begin{aligned}
\label{orlizSigma}
\mathbb{E}\bigg[ \sup_{t\in [0,T]} \int_{\mathbb{R}^3}\Big(
\vert \sigma^\varepsilon \vert^2 \chi_{2\vert \sigma^\varepsilon\vert \leq 1}
+
\vert \sigma^\varepsilon \vert^\gamma \chi_{2\vert \sigma^\varepsilon\vert > 1}\Big)\mathrm{d}x
\bigg]^p \lesssim 1
\end{aligned}
\end{equation}
uniformly in $\varepsilon>0$ and when $\gamma\geq 2$, we obtain
\begin{equation}
\begin{aligned}
\label{orlizSigmaNot}
\mathbb{E}\bigg[ \sup_{t\in [0,T]} \int_{\mathbb{R}^3}\vert \sigma^\varepsilon \vert^2\mathrm{d}x
\bigg]^p \lesssim 1
\end{aligned}
\end{equation}
uniformly in $\varepsilon>0$ and thus our final result.
\end{proof}
Also, the regularity of the (scaled) product of density fluctuation and the electric field is given in the lemma below.
\begin{lemma}
\label{denFlucEstWorsened}
For all $p\in [1,\infty)$, the following  estimates 
\begin{align}
\label{denFlucEstWorsenedA}
\mathbb{E}\bigg[\sup_{t\in [0,T]} \Vert \varepsilon^{-1}  \sigma^\varepsilon \nabla V^\varepsilon  \Vert_{W^{-l,2} (\mathbb{R}^3)}   \bigg]^p \, 
\lesssim_p 1
\end{align}
holds uniformly in $\varepsilon$ provided $l \geq 1+\frac{3}{2}$.
\end{lemma}
\begin{proof}
First of all, we can use   \eqref{densityFluctuationFunction} to rewrite the Poisson equation \eqref{elecFieldB} as $
\varepsilon^{\beta-1} \Delta V^\varepsilon  
=
\sigma^\varepsilon.$
By using the following identity 
\begin{align}
\label{idenVectorCal}
\Delta V^\varepsilon \nabla V^\varepsilon=\mathrm{div}( \nabla V^\varepsilon \otimes  \nabla V^\varepsilon) -\frac{1}{2}\nabla\vert  \nabla V^\varepsilon \vert^2,
\end{align}
it follows that
\begin{align}
\label{idenVectorCalNext}\varepsilon^{-1}
 \sigma^\varepsilon  \nabla V^\varepsilon=\varepsilon^{\beta-2} \bigg(\mathrm{div}( \nabla V^\varepsilon \otimes  \nabla V^\varepsilon) -\frac{1}{2}\nabla\vert  \nabla V^\varepsilon \vert^2 \bigg).
\end{align}
Now since for $k\geq \frac{3}{2}$, the estimate
\begin{align}
\mathbb{E} \bigg[ \sup_{t\in[0,T]}\Big\Vert \varepsilon^{\beta-2} \big\vert \nabla V^\varepsilon  \big\vert^2 \Big\Vert_{W^{-k,2}(\mathbb{R}^3)} \bigg]^p \, 
\lesssim_p 1
\end{align}
follow from  the continuous embedding $L^\infty_tL^1_x \hookrightarrow L^\infty_tW^{-k,2}_x$ and the bottom-left estimate of \eqref{energyEst1}, the bound \eqref{denFlucEstWorsenedA} holds true where $l=k+1$. 
\end{proof}
In addition, the following lemma holds true.
\begin{lemma}
\label{lem:veloAndSomething}
For all $p\in [1,\infty)$, the following  estimates 
\begin{align*}
\mathbb{E}\bigg[\int_0^T \Vert \mathbf{u}^\varepsilon \Vert_{W^{1,2}(\mathbb{R}^3)}^2   \,\mathrm{d}t\,\bigg]^p \, 
\lesssim_p 1, 
\, \quad
\mathbb{E}\bigg[\int_0^T \Vert \sigma^\varepsilon \mathbf{u}^\varepsilon \Vert_{W^{-1,2}(\mathbb{R}^3)}^2   \,\mathrm{d}t\,\bigg]^p \, 
\lesssim_p 1
\end{align*}
holds uniformly in $\varepsilon$.
\end{lemma}
\begin{proof}
First of all, if we decompose $\mathbb{R}^3$ into 
\begin{align*}
\big\{x\in \mathbb{R}^3 \,:\,2\vert\rho-1 \vert\leq 1\big\}
\quad \text{and} \quad
\big\{x\in \mathbb{R}^3 \, : \,2\vert\rho-1 \vert\geq 1\big \},
\end{align*}
then we obtain,                            
\begin{equation}
\begin{aligned}
\label{l2VelA}
\mathbb{E}\bigg[ \int_0^T\Vert \mathbf{u}^\varepsilon \Vert_{L^2(\mathbb{R}^3)}^2   \,\mathrm{d}t \bigg]^p
 &\leq
\mathbb{E}\bigg[ 2 \sup_{t\in[0,T] } \big\Vert \varrho^\varepsilon-1  \big\Vert_{L^{\overline{\gamma}}(\mathbb{R}^3)}
 \int_0^T\big\Vert \bu^\varepsilon  \big\Vert^2_{L^{\frac{2\overline{\gamma}}{\overline{\gamma}-1}}(\mathbb{R}^3)}\, \mathrm{d}t \bigg]^p
 \\&+
 \mathbb{E}\bigg[ 2 \sup_{t\in[0,T] } \big\Vert \varrho^\varepsilon \vert \bu^\varepsilon \vert^2 \big\Vert_{L^1(\mathbb{R}^3)}\bigg]^p
\end{aligned}
\end{equation}
for any $p\in[1,\infty)$. Next, 
%since we have $0<\frac{3}{2\overline{\gamma}}<1$, $0<1- \frac{3}{2\overline{\gamma}}<1$ and 
%\begin{align*}
%\frac{\overline{\gamma}-1}{2\overline{\gamma}}
%=
%\bigg(1- \frac{3}{2\overline{\gamma}} \bigg)\times \frac{1}{2} + \frac{3}{2\overline{\gamma}}\times \frac{1}{6},
%\end{align*}
we  interpolate $\bu^\varepsilon$ between $L^2(\mathbb{R}^3)$ and  $L^6(\mathbb{R}^3)$ which yields
\begin{equation}
\begin{aligned}
\label{l2VelA00}
\mathbb{E}\bigg[& 2
 \sup_{t\in[0,T] } \big\Vert \varrho^\varepsilon -1  \big\Vert_{L^{\overline{\gamma}}(\mathbb{R}^3)}
 \int_0^T\big\Vert \bu^\varepsilon  \big\Vert^2_{L^{\frac{2\overline{\gamma}}{\overline{\gamma}-1}}(\mathbb{R}^3)}\, \mathrm{d}t
 \bigg]^p
\\&\leq 2^p \Bigg\{
\Bigg( \mathbb{E}\bigg[ \int_0^T
\big\Vert  \bu^\varepsilon  \big\Vert^2_{L^2(\mathbb{R}^3)}\, \mathrm{d}t \bigg]^{p_1\left(1-\frac{3}{2\overline{\gamma}} \right) }
\Bigg)^\frac{p}{p_1}
\\&
+
\Bigg( \mathbb{E}\bigg[
 \sup_{t\in[0,T] } \big\Vert \varrho^\varepsilon -1  \big\Vert_{L^{\overline{\gamma}}(\mathbb{R}^3)}
 \bigg]^{p_2}
 \Bigg)^\frac{p}{p_2}
\Bigg( 
\mathbb{E}
 \bigg[ \int_0^T\big\Vert  \vert \nabla \bu^\varepsilon \vert^2 \big\Vert_{L^1(\mathbb{R}^3)}
\, \mathrm{d}t \bigg]^{p_3\frac{3}{2\overline{\gamma}}}
\Bigg)^\frac{p}{p_3} \Bigg\} 
\end{aligned}
\end{equation} 
uniformly in $\varepsilon$  for all $p_1,p_2,p_3 \in [1,\infty)$ such that $\frac{1}{p_1}+\frac{1}{p_2}+\frac{1}{p_3}\leq \frac{1}{p}$. Now since
\begin{align}
\Bigg( \mathbb{E}\bigg[
 \sup_{t\in[0,T] } \big\Vert \varrho^\varepsilon -1  \big\Vert_{L^{\overline{\gamma}}(\mathbb{R}^3)}
 \bigg]^{p_2}
 \Bigg)^\frac{p}{p_2}
 \lesssim_p \varepsilon^\frac{2p}{\overline{\gamma}}
\end{align}
follow from \eqref{denFlucEst},
we obtain from \eqref{l2VelA}--\eqref{l2VelA00} and the top-right estimate of \eqref{energyEst1}, the following estimate
\begin{align*}
\mathbb{E}\bigg[\int_0^T \Vert \mathbf{u}^\varepsilon \Vert_{L^2(\mathbb{R}^3)}^2   \,\mathrm{d}t\,\bigg]^p
&\lesssim_p 1
\end{align*} 
uniformly in $\varepsilon>0$. Thus, together with the last estimate in \eqref{energyEst1}, we obtain the first estimate in Lemma \ref{lem:veloAndSomething}.
To show the second estimate in Lemma \ref{lem:veloAndSomething}, we fist note that by interpolating between \eqref{velocityl6} and the first  estimate of Lemma \ref{lem:veloAndSomething} which we have just shown, we can conclude that 
\begin{equation}
\begin{aligned}
\label{velL4plusSomething}
\mathbb{E}\bigg[\int_0^T \Vert \mathbf{u}^\varepsilon \Vert_{L^4(\mathbb{R}^3)}^2   \,\mathrm{d}t\,\bigg]^p
+
\mathbb{E}\bigg[\int_0^T \Vert \mathbf{u}^\varepsilon \Vert_{L^\frac{2\overline{\gamma}}{\overline{\gamma}-1}(\mathbb{R}^3)}^2   \,\mathrm{d}t\,\bigg]^p \lesssim 1
\end{aligned}
\end{equation}
holds uniformly in $\varepsilon$. Now if $\overline{\gamma}=2$, then by using the $L^4_x$-estimate of \eqref{velL4plusSomething} above as well as \eqref{denFlucEst}, we gain that for some $p_1,p_2\in (1,\infty)$ such that $\frac{1}{p}=\frac{1}{p_1} +\frac{1}{p_2}$,
\begin{equation}
\begin{aligned}
\label{plusEst1}
\mathbb{E}\bigg[\int_0^T \Vert \sigma^\varepsilon \mathbf{u}^\varepsilon \Vert_{L^\frac{4}{3}(\mathbb{R}^3)}^2   \,&\mathrm{d}t\,\bigg]^p
\leq 
\bigg(\mathbb{E} \bigg[  \sup_{t\in[0,T]}\Vert \sigma^\varepsilon \Vert^2_{L^2(\mathbb{R}^3)} \bigg]^{p_1} \bigg)^\frac{p}{p_1}
\\&\times
\bigg(
\mathbb{E}\bigg[\int_0^T \Vert \mathbf{u}^\varepsilon \Vert_{L^4(\mathbb{R}^3)}^2   \,\mathrm{d}t\,\bigg]^{p_2} \bigg)^\frac{p}{p_2} \lesssim 1
\end{aligned}
\end{equation}
uniformly in $\varepsilon$. Similarly, if $\overline{\gamma} =\gamma$ (so that $\gamma<2$) we can use the $L^r_x$-estimate of \eqref{velL4plusSomething} where $r=\frac{2\overline{\gamma}}{\overline{\gamma}-1}$ to obtain for some $p_1,p_2\in (1,\infty)$ ,
\begin{equation}
\begin{aligned}
\label{plusEst2}
\mathbb{E}\bigg[\int_0^T \Vert \sigma^\varepsilon \mathbf{u}^\varepsilon \Vert_{L^\frac{2\gamma}{\gamma+1}(\mathbb{R}^3)}^2  &\mathrm{d}t\,\bigg]^p
\leq 
\bigg(\mathbb{E} \bigg[  \sup_{t\in[0,T]}\Vert \sigma^\varepsilon \Vert^2_{L^\gamma(\mathbb{R}^3)} \bigg]^{p_1} \bigg)^\frac{p}{p_1}
\\&\times
\bigg(
\mathbb{E}\bigg[\int_0^T \Vert \mathbf{u}^\varepsilon \Vert_{L^\frac{2\gamma}{\gamma-1}(\mathbb{R}^3)}^2   \mathrm{d}t\,\bigg]^{p_2} \bigg)^\frac{p}{p_2} \lesssim 1
\end{aligned}
\end{equation}
uniformly in $\varepsilon$. We can therefore draw our conclusion from \eqref{plusEst1}--\eqref{plusEst2} and the continuous embedding
\begin{align*}
L^2\big(0,T; L^\frac{4}{3}+ L^q(\mathbb{R}^3) \big)\hookrightarrow L^2\big(0,T; W^{-1,2}(\mathbb{R}^3) \big)
\end{align*}
which holds for every $q\geq \frac{6}{5}$. Note that $\frac{2\overline{\gamma}}{\overline{\gamma}+1} \in (\frac{6}{5}, \frac{4}{3})$ when $\overline{\gamma}=\gamma<2$.
\end{proof}
As a result of Lemma \ref{lem:veloAndSomething}, we obtain the following lemma.
\begin{lemma}
\label{lem:momAndConvec}
For all $p\in [1,\infty)$ and any ball $B_m\subset \mathbb{R}^3$ of radius $m\in \mathbb{N}$, there exists a constant $a_m>0$ such that the following  estimates 
\begin{align*}
\mathbb{E}\bigg[\sup_{t\in[0,T]} \Vert \sigma^\varepsilon  \Vert_{L^{\overline{\gamma}}(B_m)}   \bigg]^p \, 
+
\mathbb{E}\bigg[\sup_{t\in[0,T]} \Vert \varrho^\varepsilon \mathbf{u}^\varepsilon \Vert_{L^\frac{2\gamma}{\gamma+1}(B_m)}   \bigg]^p \, 
+
\mathbb{E}\bigg[\sup_{t\in[0,T]} \Vert \varrho^\varepsilon  \Vert_{L^{\gamma}(B_m)}   \bigg]^p \, 
\lesssim_p  a_m
\end{align*}
holds uniformly in $\varepsilon$ where $\overline{\gamma}=\min \{\gamma,2 \}$.
\end{lemma}
\begin{proof}
The first (local) estimate follow from the global estimate \eqref{denFlucEst}. For the second, see for example \cite[ (23) or (38)]{mensah2016existence}. The third estimate is shown in \cite[(3.7)]{breit2015incompressible} where the torus is now replaced with a ball.
\end{proof}
In order to be able to prove compactness result in the next section we recover more refined estimates on the soleinoidal and gradient part of the momentum. 
\subsection{Analysis of the solenoidal part of momentum}
\label{sec:solePart}
\begin{lemma}
\label{lem:SoleMomen}
For any $\vartheta\in (0,\frac{1}{2})$ and some $l\geq1$, the following  estimate
\begin{align*}
\mathbb{E} \left\Vert  \mathcal{P}\,( \varrho^\varepsilon\mathbf{u}^\varepsilon) \right\Vert^p_{C^\vartheta\left([0,T];W^{-l,2}(\mathbb{R}^3)  \right)}\lesssim 1
\end{align*}
holds uniformly in $\varepsilon$ for some $p\in[1,\infty)$.
\end{lemma}
\begin{proof}
Since $\varepsilon^{-2}\varrho^\varepsilon  \nabla V^\varepsilon =\varepsilon^{-1} \sigma^\varepsilon  \nabla V^\varepsilon +\varepsilon^{-2}\nabla V^\varepsilon,$
we can use \eqref{elecFieldB} and \eqref{idenVectorCal} to rewrite \eqref{momEqB} as
\begin{equation}
\begin{aligned}
\label{tightMom}
&\mathrm{d}(\varrho^\varepsilon \mathbf{u}^\varepsilon) +
\bigg[
\mathrm{div} (\varrho^\varepsilon \mathbf{u}^\varepsilon \otimes \mathbf{u}^\varepsilon) +\frac{1}{\varepsilon^2} \nabla   (\varrho^\varepsilon)^\gamma \bigg]\, \mathrm{d}t 
=\bigg[ \nu_1 \Delta \mathbf{u}^\varepsilon +(\nu_2 + \nu_1)\nabla \mathrm{div} \mathbf{u}^\varepsilon 
\\
&+ \varepsilon^{\beta-2}\mathrm{div} ( \nabla V^\varepsilon \otimes  \nabla V^\varepsilon)
-
\frac{1}{2} \varepsilon^{\beta-2}  \nabla\vert \nabla V^\varepsilon\vert^2
+
\frac{1}{\varepsilon^2}  \nabla V^\varepsilon  \bigg] \, \mathrm{d}t 
+
\mathbf{G}(\varrho^\varepsilon,  \varrho^\varepsilon\mathbf{u}^\varepsilon) \, \mathrm{d}W
\end{aligned}
\end{equation}
whose projection onto $\mathcal{P}$ is
\begin{equation}
\begin{aligned}
\label{tightMom1}
\mathrm{d}\mathcal{P}(\varrho^\varepsilon \mathbf{u}^\varepsilon) &+ \mathcal{P}
\big[
\mathrm{div} (\varrho^\varepsilon \mathbf{u}^\varepsilon \otimes \mathbf{u}^\varepsilon) - \nu_1 \Delta \mathbf{u}^\varepsilon  
- \varepsilon^{\beta-2}\mathrm{div} ( \nabla V^\varepsilon \otimes  \nabla V^\varepsilon)
 \big] \, \mathrm{d}t 
\\&= \mathcal{P}\mathbf{G}(\varrho^\varepsilon,  \varrho^\varepsilon\mathbf{u}^\varepsilon) \, \mathrm{d}W
\end{aligned}
\end{equation}
in the sense of distributions. To be more precise, for all  $\bm{\phi} \in C^\infty_{c, \mathrm{div}} (\mathbb{R}^3)$ and all $t\in [0,T]$, the following
\begin{equation}
\begin{aligned}
 \int_{\mathbb{R}^3}\varrho^\varepsilon  \mathbf{u}^\varepsilon(t) \cdot \bm{\phi} \,\mathrm{d}x   
&=   \int_{\mathbb{R}^3} \varrho ^\varepsilon_0 \mathbf{u}^\varepsilon _0 \cdot \bm{\phi} \,\mathrm{d}x  +  \int_0^t  \int_{\mathbb{R}^3} \varrho^\varepsilon  \mathbf{u}^\varepsilon \otimes \mathbf{u}^\varepsilon : \nabla \bm{\phi}\,\mathrm{d}x  \mathrm{d}s 
\\&- {\nu}_1\int_0^t  \int_{\mathbb{R}^3} \nabla \mathbf{u}^\varepsilon : \nabla \bm{\phi} \,\mathrm{d}x  \mathrm{d}s 
-\varepsilon^{\beta-2}
 \int_0^t  \int_{\mathbb{R}^3} \nabla  V^\varepsilon \otimes \nabla V^\varepsilon : \nabla \bm{\phi}\,\mathrm{d}x  \mathrm{d}s 
\\
& +
\int_0^t  \int_{\mathbb{R}^3}  \mathbf{G}(\varrho^\varepsilon ,  \varrho^\varepsilon \mathbf{u}^\varepsilon ) \cdot \bm{\phi} \,\mathrm{d}x\, \mathrm{d}W
\end{aligned}
\end{equation}
hold $\mathbb{P}$-a.s. Similar to \cite[Proposition 3.6]{breit2015incompressible},  we can now divide the above into the following
\begin{equation}
\begin{aligned}
Y^\varepsilon(t):=\mathcal{P}(\varrho^\varepsilon_0 \mathbf{u}^\varepsilon_0) &+ \int_0^t\mathcal{P}
\big[
\mathrm{div} (\varrho^\varepsilon \mathbf{u}^\varepsilon \otimes \mathbf{u}^\varepsilon) - \nu_1 \Delta \mathbf{u}^\varepsilon  
- \varepsilon^{\beta-2} \mathrm{div} ( \nabla V^\varepsilon \otimes  \nabla V^\varepsilon)
 \big] \, \mathrm{d}s 
\end{aligned}
\end{equation}
and 
\begin{align}
Z^\varepsilon(t):=\int_0^t \mathcal{P} \mathbf{G}(\varrho^\varepsilon,  \varrho^\varepsilon\mathbf{u}^\varepsilon) \, \mathrm{d}W(s).
\end{align}
Given that the embedding $L^1_x \hookrightarrow W^{-l,2}_x$ is continuous for $l>\frac{3}{2}$, we obtain for any $t_1,t_2\in[0,T]$ and $\theta >1 $,
\begin{equation}
\begin{aligned}\nonumber
&\mathbb{E}\big\Vert Y^\varepsilon(t_2) -Y^\varepsilon(t_1)  \big\Vert_{W^{-l,2}(\mathbb{R}^3)}^\theta
\\&
\lesssim\mathbb{E}\Bigg( \bigg\vert \int_{t_1}^{t_2}  \int_{\mathbb{R}^3} \big( \varrho^\varepsilon  \mathbf{u}^\varepsilon \otimes \mathbf{u}^\varepsilon 
- {\nu}_1 \nabla \mathbf{u}^\varepsilon 
-\varepsilon^{\beta-2}
  \nabla  V^\varepsilon \otimes \nabla V^\varepsilon \big) : \nabla \bm{\phi}\,\mathrm{d}x  \mathrm{d}s  \bigg\vert \Bigg)^\theta
\\&
\lesssim \mathbb{E}\Bigg( \int_{t_1}^{t_2}  \int_{\mathbb{R}^3} \vert \varrho^\varepsilon  \mathbf{u}^\varepsilon \otimes \mathbf{u}^\varepsilon \vert\,\mathrm{d}x  \mathrm{d}s 
+ {\nu}_1\int_{t_1}^{t_2}  \int_{\mathbb{R}^3}  \vert\nabla \mathbf{u}^\varepsilon\vert \,\mathrm{d}x  \mathrm{d}s 
+
\int_{t_1}^{t_2}  \int_{\mathbb{R}^3}\vert \varepsilon^{\beta-2}
  \nabla  V^\varepsilon \otimes \nabla V^\varepsilon \vert \,\mathrm{d}x  \mathrm{d}s  \Bigg)^\theta
  \\&
\lesssim \mathbb{E}\,\bigg( \int_{t_1}^{t_2}  \Big\Vert \varrho_\varepsilon \big\vert \bu^\varepsilon  \big\vert^2 \Big\Vert_{L^1(\mathbb{R}^3)}  \mathrm{d}s \bigg)^\theta
+
\mathbb{E}\,\bigg( \int_{t_1}^{t_2} \Big\Vert  \big\vert \nabla \bu^\varepsilon  \big\vert^2 \Big\Vert_{L^1(\mathbb{R}^3)}  \mathrm{d}s \bigg)^\theta
+
\mathbb{E}\,\bigg( \int_{t_1}^{t_2}  
\Big\Vert \varepsilon^{\beta-2} \big\vert \nabla V^\varepsilon  \big\vert^2 \Big\Vert_{L^1(\mathbb{R}^3)}  \mathrm{d}s   \bigg)^\theta.
\end{aligned}
\end{equation}
If we now use   \eqref{energyEst1}, then it follows that
\begin{equation}
\begin{aligned}\mathbb{E}\big\Vert Y^\varepsilon(t_2) -Y^\varepsilon(t_1)  \big\Vert_{W^{-l,2}(\mathbb{R}^3)}^\theta
\lesssim \vert t_2- t_1\vert^\frac{\theta}{2}.
\end{aligned}
\end{equation}
Also, since the noise term is of compact support, then similar to \eqref{seqNois1}, we also obtain from the continuous embedding $L^1_x \hookrightarrow W^{-l,2}_x$,  $l>\frac{3}{2}$, and the bound \eqref{noiseAssumptionWholespace2},
\begin{equation}
\begin{aligned}
\mathbb{E}&\big\Vert Z^\varepsilon(t_2) -Z^\varepsilon(t_1)  \big\Vert_{W^{-l,2}(\mathbb{R}^3)}^\theta 
\lesssim   \mathbb{E}\bigg[\int_{t_1}^{t_2} \sum_{k\in \mathbb{N}}\Vert  \mathbf{g}_k(x, \varrho^\varepsilon,  \varrho^\varepsilon \mathbf{u}^\varepsilon)\Vert_{W^{-l,2}(\mathbb{R}^3)}^{2}\bigg]^\frac{\theta}{2}
\\
&\lesssim \mathbb{E}\bigg[ \int_{t_1}^{t_2} \sum_{k\in \mathbb{N}} \bigg(\int_{K}\sqrt{\varrho^\varepsilon} \,\frac{1}{\sqrt{\varrho^\varepsilon}}\vert  \mathbf{g}_k(x, \varrho^\varepsilon,  \varrho^\varepsilon \mathbf{u}^\varepsilon)   \vert \, \mathrm{d}x\bigg)^{2}\, \mathrm{d}s \bigg]^\frac{\theta}{2}
\\
&\lesssim \mathbb{E}\bigg[ \int_{t_1}^{t_2} \bigg(\int_{K}\varrho^\varepsilon \, \mathrm{d}x\bigg)\bigg(\int_{K} \frac{1}{\varrho^\varepsilon}\sum_{k\in \mathbb{N}} \vert  \mathbf{g}_k(x, \varrho^\varepsilon,  \varrho^\varepsilon \mathbf{u}^\varepsilon)   \vert^2 \, \mathrm{d}x\bigg)\, \mathrm{d}s \bigg]^\frac{\theta}{2}
\\
&\lesssim \mathbb{E}\bigg[ \int_{t_1}^{t_2} \bigg(\int_{K}\Big[ 1+ \frac{1}{\varepsilon^2} H(\varrho^\varepsilon, 1 ) \Big]\, \mathrm{d}x\bigg)
 \bigg(\int_{K}\Big[\frac{1}{2}\varrho^\varepsilon \vert \mathbf{u}^\varepsilon \vert^2+ \varrho^\varepsilon \Big] \, \mathrm{d}x\bigg)\, \mathrm{d}s \bigg]^\frac{\theta}{2}
\\&
\lesssim \vert t_2- t_1\vert^\frac{\theta}{2}
\Bigg( 1+ \mathbb{E}\, \sup_{t\in[0,T]}\bigg\Vert \frac{1}{2}\varrho^\varepsilon\vert  \mathbf{u}^\varepsilon \vert^2 \bigg\Vert_{L^1(\mathbb{R}^3)} ^\theta
+
 \mathbb{E}\, \sup_{t\in[0,T]}\bigg\Vert \frac{1}{\varepsilon^2} H(\varrho^\varepsilon, 1 ) \bigg\Vert_{L^1(\mathbb{R}^3)}^\theta
 \Bigg)
\end{aligned}
\end{equation}
where $K \Subset \mathbb{R}^3$ is the support of the noise-term
and hence our claim follow from the Kolmogorov continuity criterion since we have the estimate \eqref{energyEst1}.
%  and \eqref{denFlucEst}.
\end{proof}

\subsection{Analysis of the gradient part of momentum: acoustic system}
\label{sec:gradPart}
The aim of this section is to analyse the gradient part of the momentum and, in particular, to show that any family of such vector fields converges strongly to zero. The weak convergence of the gradient part of momentum  is strictly related to the so called plasma oscillation and  to the propagation  of the acoustic waves. Therefore the first step is to recover the equations satisfied by the acoustic waves and to investigate  the related dispersive properties which will allow us to estimate the gradient part of the momentum and the density fluctuation. In order to achieve this, we first note that as a result of \eqref{HVarrhoOverline}, \eqref{densityFluctuationFunction} and \eqref{idenVectorCalNext}, for all $\psi \in C^\infty_c (\mathbb{R}^3)$ and $\bm{\phi} \in C^\infty_c (\mathbb{R}^3)$ and all $t\in [0,T]$, we can rewrite the system \eqref{eq:distributionalSol} as
\begin{equation}
\begin{aligned}
\label{eq:distributionalSolEqui0}
 \varepsilon&\int_{\mathbb{R}^3}\sigma^\varepsilon  \psi \,\mathrm{d}x   = \varepsilon  \int_{\mathbb{R}^3} \sigma^\varepsilon _0 \, \psi \mathrm{d}x  +  \int_0^t  \int_{\mathbb{R}^3} \varrho^\varepsilon  \mathbf{u}^\varepsilon  \cdot\nabla \psi \,\mathrm{d}x  \mathrm{d}s,
\\
 \varepsilon&\int_{\mathbb{R}^3}\varrho^\varepsilon  \mathbf{u}^\varepsilon \cdot \bm{\phi} \,\mathrm{d}x   =  \varepsilon \int_{\mathbb{R}^3} \varrho ^\varepsilon_0 \mathbf{u}^\varepsilon _0 \cdot \bm{\phi} \,\mathrm{d}x  +\gamma \int_0^t  \int_{\mathbb{R}^3} \sigma^{\varepsilon} \, \mathrm{div}\, \bm{\phi}\,\mathrm{d}x  \mathrm{d}s  
+
\int_0^t  \int_{\mathbb{R}^3} \frac{1}{\varepsilon}\nabla V^\varepsilon  
\cdot
\bm{\phi} \,\mathrm{d}x    \mathrm{d}s
\\&+  \varepsilon \int_0^t  \int_{\mathbb{R}^3} (\mathbb{F}^{\varepsilon}_1 +\mathbb{F}^{\varepsilon}_2) : \nabla \bm{\phi}\,\mathrm{d}x  \mathrm{d}s 
+ \varepsilon\int_0^t  \int_{\mathbb{R}^3} (F^{\varepsilon}_1 +F^{\varepsilon}_2) \, \mathrm{div}\, \bm{\phi}\,\mathrm{d}x  \mathrm{d}s 
\\&+
 \varepsilon\int_0^t  \int_{\mathbb{R}^3}  \mathbf{G}(\varrho^\varepsilon ,  \varrho^\varepsilon\mathbf{u}^\varepsilon ) \cdot \bm{\phi} \,\mathrm{d}x\, \mathrm{d}W
\end{aligned}
\end{equation}
$\mathbb{P}$-a.s. where
\begin{align}
&\mathbb{F}^{\varepsilon}_1 := 
 (\varrho^\varepsilon\mathbf{u}^\varepsilon \otimes \mathbf{u}^\varepsilon)
-
\varepsilon^{\beta-2} (\nabla V^\varepsilon \otimes \nabla V^\varepsilon),
&\mathbb{F}^{\varepsilon}_2 := 
-
\nu_1\nabla \mathbf{u}^\varepsilon,  \label{boldF}
\\
&F^{\varepsilon}_1 := 
\frac{1}{\varepsilon^2} (\gamma -1)H ({\varrho}^{\varepsilon}, 1) 
 +
\varepsilon^{\beta-2}\frac{1}{2}
  \vert \nabla V^{\varepsilon} \vert^2,
  &F^{\varepsilon}_2 := 
 - (\nu_1 + \nu_2) \mathrm{div}\,\mathbf{u}^\varepsilon  \label{normalF}
\end{align}
are such that by \eqref{energyEst1}, we have
\begin{equation}
\begin{aligned}
\label{estF3}
\mathbb{E}\bigg[\int_0^T \Vert \mathbb{F}^{\varepsilon}_1 \Vert^2_{ L^1 (\mathbb{R}^3)} \,\mathrm{d}t \bigg]^p  \lesssim_p1, \qquad
\mathbb{E}\bigg[\int_0^T \Vert \mathbb{F}^{\varepsilon}_2 \Vert^2_{ L^2 (\mathbb{R}^3)} \,\mathrm{d}t \bigg]^p  \lesssim_p1
\\
\mathbb{E}\bigg[\int_0^T \Vert F^{\varepsilon}_1 \Vert^2_{ L^1 (\mathbb{R}^3)} \,\mathrm{d}t \bigg]^p  \lesssim_p1, \qquad
\mathbb{E}\bigg[\int_0^T \Vert F^{\varepsilon}_2 \Vert^2_{ L^2 (\mathbb{R}^3)} \,\mathrm{d}t \bigg]^p  \lesssim_p1
\end{aligned}
\end{equation}
uniformly in $\varepsilon$.
\\
Our aim now is to analyse the oscillating waves generated in the system  \eqref{eq:distributionalSolEqui0}  by projecting the system onto its (weakly) curl-free part or  gradient part. For this reason, it suffices to consider test functions $\psi \in C^\infty_c (\mathbb{R}^3)$ and $\nabla \phi =\bm{\phi} \in C^\infty_c (\mathbb{R}^3)$ such that  for all $t\in [0,T]$, we have
\begin{equation}
\begin{aligned}
\label{eq:distributionalSolEqui}
  \varepsilon^{\beta+1}&\int_{\mathbb{R}^3}\sigma^\varepsilon  \psi \,\mathrm{d}x   =    \varepsilon^{\beta+1}\int_{\mathbb{R}^3} \sigma^\varepsilon _0 \, \psi \mathrm{d}x  
  +
  \varepsilon^{\beta}  \int_0^t  \int_{\mathbb{R}^3} \varrho^\varepsilon  \mathbf{u}^\varepsilon  \cdot\nabla \psi \,\mathrm{d}x  \mathrm{d}s,
\\
\varepsilon^{\beta+1}&\int_{\mathbb{R}^3}\varrho^\varepsilon  \mathbf{u}^\varepsilon \cdot \nabla \phi \,\mathrm{d}x   
 =  
 \varepsilon^{\beta+1}\int_{\mathbb{R}^3} \varrho ^\varepsilon_0 \mathbf{u}^\varepsilon _0 \cdot \nabla \phi \,\mathrm{d}x  
 + \gamma \varepsilon^{\beta}
 \int_0^t  \int_{\mathbb{R}^3} \sigma^{\varepsilon} \, \Delta \phi\,\mathrm{d}x  \mathrm{d}s 
 \\&+
\int_0^t  \int_{\mathbb{R}^3} \nabla \Delta^{-1}\sigma^\varepsilon \cdot 
\nabla\phi \,\mathrm{d}x    \mathrm{d}s
+  \varepsilon^{\beta+1} \int_0^t  \int_{\mathbb{R}^3} (\mathbb{F}^{\varepsilon}_1 +\mathbb{F}^{\varepsilon}_2) : \nabla^2 \phi\,\mathrm{d}x  \mathrm{d}s 
 \\&+ \varepsilon^{\beta+1}\int_0^t  \int_{\mathbb{R}^3} (F^{\varepsilon}_1+F^{\varepsilon}_2) \, \Delta \phi\,\mathrm{d}x  \mathrm{d}s 
+
\varepsilon^{\beta+1}\int_0^t  \int_{\mathbb{R}^3}  \mathbf{G}(\varrho^\varepsilon ,  \varrho^\varepsilon\mathbf{u}^\varepsilon ) \cdot \nabla \phi \,\mathrm{d}x\mathrm{d}W
\end{aligned}
\end{equation}
$\mathbb{P}$-a.s. where  we have used the relation
\begin{align}
\int_0^t \int_{\mathbb{R}^3}  \frac{1}{\varepsilon}\nabla V^\varepsilon\cdot \nabla  
\phi \,\mathrm{d}x    \mathrm{d}s
=
-
\frac{1}{\varepsilon^\beta}
\int_0^t  \int_{\mathbb{R}^3} \sigma^\varepsilon  
\phi \,\mathrm{d}x    \mathrm{d}s
=
\frac{1}{\varepsilon^\beta}
\int_0^t  \int_{\mathbb{R}^3} \nabla \Delta^{-1}\sigma^\varepsilon \cdot 
\nabla\phi \,\mathrm{d}x    \mathrm{d}s
\end{align}
which follows from the Poisson equation \eqref{elecFieldB}.
If we now set $\Psi^\varepsilon  = \Delta^{-1} \mathrm{div} (    \varrho^\varepsilon  \mathbf{u}^\varepsilon )$ so that $\nabla \Psi^\varepsilon   = \mathcal{Q} (    \varrho^\varepsilon  \mathbf{u}^\varepsilon )$, then we observe that solving \eqref{eq:distributionalSolEqui} is exactly the same as solving the following stochastic inhomogeneous Klein--Gordon system of equation
\begin{equation}
\begin{aligned}
\label{inhomogenousKG}
&\varepsilon^{\beta+1}\mathrm{d} \sigma^\varepsilon + \varepsilon^{\beta}\Delta {\Psi}^\varepsilon \, \mathrm{d}t =0, 
\\
& \varepsilon^{\beta+1}\mathrm{d}\nabla{\Psi}^\varepsilon  +  \big(\gamma\varepsilon^{\beta} \nabla - \nabla \Delta^{-1} \big) \sigma^\varepsilon\,\mathrm{d}t 
= \varepsilon^{\beta+1}\big[ \mathcal{Q}\,\mathrm{div}(\mathbb{F}^{\varepsilon}_1 +\mathbb{F}^{\varepsilon}_2)
+
\nabla (F^\varepsilon_1 +F^\varepsilon_2)  \big] \, \mathrm{d}t 
\\
&\qquad\qquad\qquad\qquad\qquad\quad\qquad \qquad\quad +   \varepsilon^{\beta+1} \mathcal{Q}\mathbf{G}(\varrho^\varepsilon,  \varrho^\varepsilon\mathbf{u}^\varepsilon) \, \mathrm{d}W,
\end{aligned}
\end{equation}
in the sense of distributions.
Our goal now is to derive dispersive estimates for the above system. To do this, we first consider the homogeneous part of \eqref{inhomogenousKG} and rescale time  for simplicity. To be more precise, we consider the following scaled (deterministic)\footnote{Here, by `deterministic', we mean that there is no stochastic forcing term in the equation. The functions can however depend on a random parameter.} homogeneous Klein--Gordon equation
\begin{equation}
\begin{aligned}
\label{homogeAcoustic}
\mathrm{d}\sigma^\varepsilon + \varepsilon^{\beta} \Delta {\Psi}^\varepsilon \,\mathrm{d}t   =  0,&   \\
 \mathrm{d}\nabla{\Psi}^\varepsilon  +  \big(\gamma\varepsilon^{\beta} \nabla - \nabla \Delta^{-1} \big) \sigma^\varepsilon\,\mathrm{d}t   =    0,&
 \\
\sigma^\varepsilon (0) = \sigma^\varepsilon_0 ; \quad \nabla{\Psi}^\varepsilon(0) = \nabla{\Psi}^\varepsilon_0&
\end{aligned}
\end{equation}
where $\Psi^\varepsilon  = \Delta^{-1} \mathrm{div} (    \varrho ^\varepsilon \mathbf{u}^\varepsilon )$ is such  that $\nabla \Psi^\varepsilon   = \mathcal{Q} (    \varrho^\varepsilon  \mathbf{u}^\varepsilon)$ and $\sigma^\varepsilon$ are a solution pair of \eqref{homogeAcoustic} given explicitly by
\begin{equation}
\begin{aligned}
\label{solution}
\nabla{\Psi}^\varepsilon (\cdot,t)  &=  \frac{1}{2}\exp\big(i\sqrt{ \varepsilon^{\beta}(1-\gamma\varepsilon^{\beta}\Delta)}\,t\big)\left( \nabla{\Psi}^\varepsilon_0(\cdot)   -\frac{i\sqrt{\varepsilon^{\beta}(1-\gamma\varepsilon^{\beta}\Delta)}}{\varepsilon^\beta} \,\nabla \Delta^{-1} \sigma^\varepsilon_0(\cdot)   \right)    
\\
&+   \frac{1}{2}\exp\big(-i\sqrt{\varepsilon^{\beta}(1-\gamma\varepsilon^{\beta}\Delta)}\,t\big)\left( \nabla{\Psi}^\varepsilon_0(\cdot)    +  \frac{i\sqrt{\varepsilon^{\beta}(1-\gamma\varepsilon^{\beta}\Delta)}}{\varepsilon^\beta} \,\nabla \Delta^{-1} \sigma^\varepsilon_0(\cdot)  \right),
\\
\sigma^\varepsilon(\cdot, t)  &=  \frac{1}{2}\exp\big(i\sqrt{\varepsilon^{\beta}(1-\gamma\varepsilon^{\beta}\Delta)}\,t\big)\left( \sigma^\varepsilon_0(\cdot)    +  \frac{i \varepsilon^\beta \Delta}{\sqrt{\varepsilon^{\beta}(1-\gamma\varepsilon^{\beta}\Delta)}}{\Psi}^\varepsilon_0(\cdot)   \right)   
\\
&+ \frac{1}{2}\exp\big(-i\sqrt{\varepsilon^{\beta}(1-\gamma\varepsilon^{\beta}\Delta)}\,t\big)\left( \sigma^\varepsilon_0(\cdot)    - \frac{i\varepsilon^\beta \Delta}{\sqrt{\varepsilon^{\beta}(1-\gamma\varepsilon^{\beta}\Delta)}} {\Psi}^\varepsilon_0(\cdot)  \right).
\end{aligned}
\end{equation}
%{\color{blue}
%As would be shown shortly, analysing the Klein--Gordon equation \eqref{homogeAcoustic} rather than the wave equation as was done in \cite{donatelli2008quasineutral}, will lead to better dispersive estimates by capturing both oscillatory and dispersive behaviours present in our system. Consequently, we also obtain a better bound for $\beta$ in \eqref{valuebeta} required for the convergence to zero of the gradient part of the velocity and momentum fields. In particular, we will show that for any $\delta>0$, the aforementioned convergence results will hold true provided that $\beta\in (0,\frac{1}{2+\delta})$, see in particular, Lemma \ref{lem:gradVelocity} below.
%\\
%}
Before we continue, we also remark that by substituting the first equation of \eqref{homogeAcoustic} into the second and taking the divergence of the resulting equation, one obtains the dispersive equation satisfied by the density fluctuation,
\begin{align}
\label{kleinGordonSecondOrder}
\partial_{tt}\sigma^\varepsilon-\varepsilon^\beta \big(\gamma\varepsilon^{\beta} \Delta - 1 \big) \sigma^\varepsilon=0
\end{align}
which is more reminiscent of the `standard' homogeneous Klein--Gordon equation especially if $\gamma=\varepsilon=1$. 
Since our ultimate goal is to add a stochastic forcing term to the Klein--Gordon equation, its first-order form \eqref{homogeAcoustic} rather than its second-order form \eqref{kleinGordonSecondOrder} is appropriate.
\begin{remark}
We want to point out that compared to the wave equation, the  Klein--Gordon equation in both its formulations \eqref{kleinGordonSecondOrder} or \eqref{homogeAcoustic}  has the property  to fully and accurately describe the properties of  our system. Indeed, as would be shown shortly, analysing the Klein--Gordon equation \eqref{homogeAcoustic} rather than the wave equation as was done in \cite{donatelli2008quasineutral}, will lead to better dispersive estimates by capturing both oscillatory  (plasma oscillations) and dispersive (acoustic waves) behaviours present in our system. Consequently, we also obtain a better bound for $\beta$ in \eqref{valuebeta} required for the convergence to zero of the gradient part of the velocity and momentum fields and to perform the convergence analysis for the electric field. In particular, we will show that for any $\delta>0$, the aforementioned convergence results will hold true provided that $\beta\in (0,\frac{1}{2+\delta})$ (compare with the values of $\beta$ obtained in \cite[Theorem 3.3]{donatelli2008quasineutral}).
\end{remark}
Let us first state and prove the following $L^2_x-L^2_x$ Strichartz estimate for any family $(\nabla\Psi^\varepsilon, \sigma^\varepsilon)_{\varepsilon>0}$ of solution pair of \eqref{homogeAcoustic}.
\begin{proposition}
\label{prop:expBound}
Let $\beta>0$ be a constant. For any $p\in[1,\infty)$ and  $r\in [1,\infty]$, the solution pair $(\nabla\Psi^\varepsilon, \sigma^\varepsilon)$ of \eqref{homogeAcoustic} satisfy the estimate
\begin{align*}
&\mathbb{E} \, \Vert  \nabla \Psi^\varepsilon \Vert_{ L^r\big(0,T;L^2(\mathbb{R}^3) \big)}^p
+
\mathbb{E} \, \Vert  \sigma^\varepsilon \Vert_{ L^r\big(0,T;L^2(\mathbb{R}^3) \big)}^p 
\lesssim \frac{1}{\varepsilon^{3\beta}}\bigg(
\mathbb{E} \, \Vert \sigma_0^\varepsilon \Vert_{L^2(\mathbb{R}^3)}^p
+
\mathbb{E} \, \Vert \nabla \Psi_0^\varepsilon \Vert_{L^2(\mathbb{R}^3)}^p\bigg)
\end{align*}
uniformly in $\varepsilon$ for $\sigma_0^\varepsilon \in L^p(\Omega; L^2(\mathbb{R}^3))$ and $\nabla \Psi_0^\varepsilon \in L^p(\Omega; L^2(\mathbb{R}^3))$.
\end{proposition}
\begin{proof}
First of all, we observe that by considering the solution pair $(\nabla\Psi^\varepsilon, \sigma^\varepsilon)$ given by \eqref{solution} in Fourier space yields the following
\begin{equation}
\begin{aligned}
\label{solutionFourier}
i\xi_i \hat{\Psi}^\varepsilon &(\xi,t)  =  \frac{1}{2}\exp\big(i\sqrt{ \varepsilon^\beta(1+\gamma\varepsilon^\beta\vert \xi \vert^2)}\,t\big)\left( i \xi_i\hat{\Psi}^\varepsilon_0(\xi)   +\frac{i\sqrt{\varepsilon^\beta(1+\gamma\varepsilon^\beta\vert \xi \vert^2) }}{\varepsilon^\beta} \,i\xi_i \vert \xi\vert^{-2} \hat{\sigma}^\varepsilon_0(\xi)   \right)    
\\
&+   \frac{1}{2}\exp\big(-i\sqrt{\varepsilon^\beta(1+\gamma\varepsilon^\beta\vert \xi \vert^2)}\,t\big)\left( i\xi_i \hat{\Psi}^\varepsilon_0(\xi)    -  \frac{i\sqrt{\varepsilon^\beta(1+\gamma\varepsilon^\beta\vert \xi \vert^2)}}{\varepsilon^\beta} \,i\xi_i \vert \xi\vert^{-2} \hat{\sigma}^\varepsilon_0(\xi)  \right),
\\
\hat{\sigma}^\varepsilon&(\xi, t)  =  \frac{1}{2}\exp\big(i\sqrt{\varepsilon^\beta(1+\gamma\varepsilon^\beta\vert \xi \vert^2)}\,t\big)\left( \hat{\sigma}_0^\varepsilon(\xi)    -  \frac{i \varepsilon^\beta\vert \xi\vert^2}{\sqrt{\varepsilon^\beta(1+\gamma\varepsilon^\beta\vert \xi \vert^2)}}\hat{\Psi}_0^\varepsilon(\xi)   \right)   
\\
&+ \frac{1}{2}\exp\big(-i\sqrt{\varepsilon^\beta(1+\gamma\varepsilon^\beta\vert \xi \vert^2)}\,t\big)\left( \hat{\sigma}_0^\varepsilon(\xi)    + \frac{i \varepsilon^\beta\vert \xi\vert^2}{\sqrt{\varepsilon^\beta(1+\gamma\varepsilon^\beta\vert \xi \vert^2)}} \hat{\Psi}_0^\varepsilon(\xi)  \right).
\end{aligned}
\end{equation}
In order to estimate $\hat{\Psi}^\varepsilon$, $\hat{\sigma}^\varepsilon$, we have to distinguish between high and low frequencies. For simplicity we set 1 to be the constant for the  upper (lower) boundary of the   low (high) frequencies.
We recall that by Plancherel's theorem,
\begin{align}
\Vert \sigma^\varepsilon(x,\cdot) \Vert_{L^2(\mathbb{R}^3)}^2 = \int_{\mathbb{R}^3} \vert\hat{\sigma}^\varepsilon(\xi,\cdot) \vert^2 \, \mathrm{d}\xi
\end{align}
and since $\vert \exp\big(\pm i\sqrt{\varepsilon^\beta(1+\gamma\varepsilon^\beta\vert \xi \vert^2)}\,t\big) \vert=1$ holds independently of $t>0$, it follows that
\begin{equation}
\begin{aligned}
\label{xihat1}
\vert\hat{\sigma}^\varepsilon(\xi,t) \vert^2 \lesssim
\vert  \hat{\sigma}_0^\varepsilon(\xi)\vert^2
+
m^\varepsilon(\xi)
\vert i\xi_i \hat{\Psi}_0^\varepsilon(\xi)\vert^2
\end{aligned}
\end{equation}
where $m^\varepsilon(\xi)$ is the  Fourier multiplier given by
\begin{align}
m^\varepsilon(\xi) 
&=\frac{\varepsilon^{\beta}\vert \xi \vert^2}{1+\gamma\varepsilon^{\beta}\vert \xi \vert^2}.
\end{align}
Now given that $\varepsilon>0$,  $\varepsilon\to0$ and $\gamma>\frac{3}{2}$, we have the following:
\begin{itemize}
\item[A :] if $0<\varepsilon^\beta< \vert \xi \vert\leq1$, then $0< m^\varepsilon(\xi)
\leq \frac{1}{\gamma+1}< \frac{2}{5}$,
\item[B :] if $0<\varepsilon^\beta< 1< \vert \xi \vert < \infty$, then
$0< m^\varepsilon(\xi)
<\frac{1}{\gamma}< \frac{2}{3}$
\end{itemize}
where cases $\mathrm{A}$  corresponds to the low-frequency regime and case $\mathrm{B}$ is the high-frequency regime. In all cases, we see that the multiplier $m^\varepsilon(\xi)$ is bounded uniformly in $\varepsilon>0$ and so it follow from \eqref{xihat1} that for any $p\in[1,\infty)$ and  $r\in [1,\infty]$,
\begin{align}
\label{xihat2}
\mathbb{E} \, \Vert  \sigma^\varepsilon \Vert_{ L^r\big(0,T;L^2(\mathbb{R}^3) \big)}^p 
\lesssim
\mathbb{E} \, \Vert \sigma_0^\varepsilon \Vert_{L^2(\mathbb{R}^3)}^p
+
\mathbb{E} \, \Vert \nabla \Psi_0^\varepsilon \Vert_{L^2(\mathbb{R}^3)}^p
\end{align}
uniformly in $\varepsilon$ for $\sigma_0^\varepsilon \in L^p(\Omega; L^2(\mathbb{R}^3))$ and $\nabla \Psi_0^\varepsilon \in L^p(\Omega; L^2(\mathbb{R}^3))$. Now similar to \eqref{xihat1}, one observes that
\begin{equation}
\begin{aligned}
\label{psihat1}
\vert i\xi_i\hat{\Psi}^\varepsilon(\xi,t) \vert^2 \lesssim
\vert i\xi_i\hat{\Psi}_0^\varepsilon(\xi)\vert^2
+
\frac{1}{\varepsilon^\beta}
n^\varepsilon(\xi)
\vert \hat{\sigma}_0^\varepsilon(\xi)\vert^2
\end{aligned}
\end{equation}
where $n^\varepsilon(\xi)$ is the  Fourier multiplier given by
\begin{align}
n^\varepsilon(\xi)
=\frac{\varepsilon^\beta}{m^\varepsilon(\xi)} 
=\frac{1+\gamma\varepsilon^{\beta}\vert \xi \vert^2}{\vert \xi \vert^2}.
\end{align}
Therefore, taking into account that $\varepsilon>0$,  $\varepsilon\to0$, we have the following:
\begin{itemize}
\item[A :] if $0<\varepsilon^\beta< \vert \xi \vert\leq1$, then $0\leq  n^\varepsilon(\xi)
<\frac{\gamma+1}{\varepsilon^{2\beta}}$,
\item[B :] if $0<\varepsilon^\beta< 1< \vert \xi \vert< \infty$, then
$0 < n^\varepsilon(\xi) <\frac{\gamma+1}{\varepsilon^{2\beta}}$.
\end{itemize}
We can thus conclude from \eqref{psihat1} that
\begin{equation}
\begin{aligned}
\label{psihat1a}
\vert i\xi_i\hat{\Psi}^\varepsilon(\xi,t) \vert^2 \lesssim_\gamma
\vert i\xi_i\hat{\Psi}_0^\varepsilon(\xi)\vert^2
+
\frac{1}{\varepsilon^{3\beta}}
\vert \hat{\sigma}_0^\varepsilon(\xi)\vert^2.
\end{aligned}
\end{equation}
Just as we did for \eqref{xihat2}, we can conclude from \eqref{psihat1a} that for any $p\in[1,\infty)$ and  $r\in [1,\infty]$,
\begin{align}
\label{psihat2}
\mathbb{E} \, \Vert \nabla \Psi^\varepsilon \Vert_{ L^r\big(0,T;L^2(\mathbb{R}^3) \big)}^p 
\lesssim_\gamma \frac{1}{\varepsilon^{3\beta}} \bigg(
\mathbb{E} \, \Vert \sigma_0^\varepsilon \Vert_{L^2(\mathbb{R}^3)}^p
+
\mathbb{E} \, \Vert \nabla \Psi_0^\varepsilon \Vert_{L^2(\mathbb{R}^3)}^p \bigg)
\end{align}
uniformly in $\varepsilon>0$. Summing \eqref{xihat2} and \eqref{psihat2} finishes the proof.
%\begin{remark}
%Note that the Fourier multipliers $m^\varepsilon$ and $n^\varepsilon$ remains bounded when we replace $1$ in the various cases by any constant $c>0$. In order words, the boundedness of the multipliers is independent of where the low frequencies ends and the high frequencies start (or vice versa).
%\end{remark}
\end{proof}

\begin{lemma}
\label{lem:GradMomen}
Let $\beta>0$ be a constant. Then for any $\delta>0$ and for any $p\in [1,2]$, the following  estimate
\begin{align}
\label{gradMomEst}
\mathbb{E}\bigg[\int_0^T \big\Vert \mathcal{Q}[\varrho^\varepsilon \mathbf{u}^\varepsilon]_\kappa \big\Vert_{L^2(\mathbb{R}^3)}^2   \,\mathrm{d}t\,\bigg]^p \, 
\lesssim_p \varepsilon^{1-(2+\delta)\beta}, 
\end{align}
holds uniformly in $\varepsilon>0$ for some  kernel $\kappa>0$. Furthermore,
\begin{align}
\label{gradMomZero}
\mathcal{Q}(\varrho^\varepsilon \mathbf{u}^\varepsilon)
\rightarrow 0
\quad
\text{in}
\quad
L^p \big( \Omega; L^2\big( 0,T; L^\frac{2\gamma}{\gamma+1}_{\mathrm{loc}}(\mathbb{R}^3) \big)\big)
\end{align}
holds as $\varepsilon \rightarrow0$ provided $0<\beta<\frac{1}{2+\delta}$.
\end{lemma}
\begin{proof}
First of all, let us rewrite the  system \eqref{inhomogenousKG} as follows, 
\begin{equation}
\begin{aligned}
\label{acousticSPD111LM}
\varepsilon^{\beta+1} \mathrm{d}
\begin{bmatrix}
      \sigma^\varepsilon      \\[0.3em]
 \nabla {\Psi}^\varepsilon 
\end{bmatrix}  
&=
\mathcal{A}
\begin{bmatrix}
    \sigma^\varepsilon        \\[0.3em]
\nabla {\Psi}^\varepsilon  
\end{bmatrix}  \mathrm{d}t 
-\varepsilon^{\beta+1}
\begin{bmatrix}
       0        \\[0.3em]
       \mathcal{Q}\,\mathrm{div}(\mathbb{F}^{\varepsilon}_1 +\mathbb{F}^{\varepsilon}_2)
+
\nabla (F^\varepsilon_1 +F^\varepsilon_2) 
\end{bmatrix}  \mathrm{d}t
\\&+\varepsilon^{\beta+1}
\begin{bmatrix}
       0        \\[0.3em]
      \mathcal{Q}\mathbf{G}^\varepsilon  
\end{bmatrix}  \mathrm{d}{W}_t
\end{aligned}
\end{equation}
where $\mathbf{G}^\varepsilon  :=
 \mathbf{G}( {\varrho}^\varepsilon , {\varrho}^\varepsilon   {\mathbf{u}}^\varepsilon )$ and where
\begin{align}
\label{KGoperator}
\mathcal{A} := -
\begin{bmatrix}
       0 & \varepsilon^{\beta}\mathrm{div}       \\[0.3em]
       \gamma\varepsilon^{\beta} \nabla - \nabla \Delta^{-1} & 0
\end{bmatrix}.
\end{align}
is an infinitesimal generator of a $C^0$-semigroup $S(t)$ with domain
\begin{align*}
\mathrm{Dom}(\mathcal{A})=\big\{ [\sigma, \nabla \Psi]^T \, : \, \sigma\in W^{1,2}(\mathbb{R}^3), \, \nabla \Psi , \Delta \Psi \in L^2(\mathbb{R}^3) \big\}
% \subset \mathcal{H}
.
\end{align*}
%where
%$\mathcal{H}= L^2(\mathbb{R}^d) \otimes L^2_\mathcal{Q}(\mathbb{R}^d)$. 
Given the properties of $\mathcal{A}$ and the fact that the noise operator is Hilbert--Schmidt, it follows from \cite[Theorems 6.5, 6.7, 7.2]{da2014stochastic} that a mild solution
\begin{equation}
\begin{aligned}
\label{mildAcoustic}
\begin{bmatrix}
     \sigma^\varepsilon         \\[0.3em]
      \nabla {\Psi}^\varepsilon 
\end{bmatrix} (t)  
&=
S\bigg(\frac{t}{\varepsilon^{\beta+1}}\bigg)
\begin{bmatrix}
     \sigma_0^\varepsilon        \\[0.3em]
       \nabla {\Psi}_0^\varepsilon   
\end{bmatrix}
-
\int_0^t 
S\bigg(\frac{t-s}{\varepsilon^{\beta+1}}\bigg)
\Bigg\{
\begin{bmatrix}
       0        \\[0.3em]
       \mathcal{Q} \mathrm{div}(\mathbb{F}^\varepsilon_1 + \mathbb{F}^\varepsilon_2) 
+
\nabla  (F^\varepsilon_1 +F^\varepsilon_2) 
\end{bmatrix}  
\Bigg\}\mathrm{d}s
\\&+
\int_0^t 
S\bigg(\frac{t-s}{\varepsilon^{\beta+1}}\bigg)
\begin{bmatrix}
       0        \\[0.3em]
      \mathcal{Q}\mathbf{G}^\varepsilon 
\end{bmatrix}  \mathrm{d}{W}_s
\end{aligned}
\end{equation}
of \eqref{acousticSPD111LM} exists and this mild solution is identical to the weak solution which one can verify to be  \eqref{eq:distributionalSolEqui}. In \eqref{mildAcoustic},
\begin{equation}
\label{semigroup}
S\left(t\right)
\begin{bmatrix}
       \sigma_0^\varepsilon  (\cdot)       \\[0.3em]
       \nabla{\Psi}_0^\varepsilon  (\cdot)
\end{bmatrix} 
=
\begin{bmatrix}
       \sigma^\varepsilon  (\cdot, t)     \\[0.3em]
     \nabla  {\Psi}^\varepsilon (\cdot, t)
\end{bmatrix},
\end{equation}
given explicitly by  the pair \eqref{solution} is the solution to the homogeneous  problem \eqref{homogeAcoustic}.

To obtain proper estimates for the solution to \eqref{mildAcoustic}, we first regularise. Since \eqref{mildAcoustic} is linear in $\Psi^\varepsilon, \mathbb{F}^\varepsilon, F^\varepsilon$ and $\mathbf{G}^\varepsilon$, by convolution with the usual mollifier $\wp_\kappa$ and the use of Proposition \ref{prop:expBound} (with the choice of $p=r=2$), we obtain 
\begin{equation}
\begin{aligned}
\label{schrit1}
&\mathbb{E} \, 
\Bigg\Vert
S( t )
\begin{bmatrix}
        [\sigma_0^\varepsilon ]_\kappa       \\[0.3em]
       \nabla[{\Psi}_0 ^\varepsilon]_\kappa
\end{bmatrix} 
\Bigg\Vert^2_{L^2(0,T;L^2(\mathbb{R}^3))}
\lesssim \frac{1}{\varepsilon^{3\beta}}
\mathbb{E} \, 
\Bigg\Vert
\begin{bmatrix}
        [\sigma_0^\varepsilon]_\kappa         \\[0.3em]
       \nabla [ {\Psi}_0^\varepsilon]_\kappa
\end{bmatrix} 
\Bigg\Vert^2_{L^2(\mathbb{R}^3)}
\end{aligned}
\end{equation}
uniformly in $\varepsilon>0$ so that after rescaling $t/\varepsilon^{\beta+1} \mapsto t$, recall \eqref{acousticSPD111LM}, we obtain
\begin{equation}
\begin{aligned}
\label{schrit2}
&\mathbb{E} \, 
\Bigg\Vert
S\bigg( \frac{t}{\varepsilon^{\beta+1}} \bigg)
\begin{bmatrix}
        [\sigma_0^\varepsilon]_\kappa        \\[0.3em]
       \nabla[{\Psi}_0 ^\varepsilon]_\kappa
\end{bmatrix} 
\Bigg\Vert^2_{L^2(0,T;L^2(\mathbb{R}^3))}
\lesssim \varepsilon^{1-2\beta}
\mathbb{E} \, 
\Bigg\Vert
\begin{bmatrix}
       [ \sigma_0^\varepsilon]_\kappa         \\[0.3em]
       \nabla [ {\Psi}_0^\varepsilon]_\kappa
\end{bmatrix} 
\Bigg\Vert^2_{L^2(\mathbb{R}^3)}
\end{aligned}
\end{equation}
with a constant independent of $\varepsilon>0$.
To treat the deterministic retarded operator on the right-hand side of \eqref{mildAcoustic}, we use the semigroup property and a similar bound as \eqref{schrit2} to get
\begin{equation}
\begin{aligned}
\label{estFa}
&\mathbb{E} \, 
\Bigg\Vert
\int_0^t
S\bigg( \frac{t-s}{\varepsilon^{\beta+1}} \bigg)
\Bigg\{
\begin{bmatrix}
       0        \\[0.3em]
       \mathcal{Q} \mathrm{div}\big([\mathbb{F}^\varepsilon_1]_\kappa + [\mathbb{F}^\varepsilon_2]_\kappa \big)
+
\nabla \big([ F^\varepsilon_1]_\kappa + [ F^\varepsilon_2]_\kappa \big)
\end{bmatrix}  
\Bigg\}\mathrm{d}s
\Bigg\Vert^2_{L^2(0,T;L^2(\mathbb{R}^3))}
\\
&=\mathbb{E} \, 
\bigg\Vert
\int_0^t
S\bigg( \frac{t}{\varepsilon^{\beta+1}} \bigg)S\bigg( \frac{-s}{\varepsilon^{\beta+1}} \bigg)
\big\{
       \mathcal{Q} \mathrm{div}\big([\mathbb{F}^\varepsilon_1]_\kappa + [\mathbb{F}^\varepsilon_2]_\kappa \big)
+
\nabla  \big([ F^\varepsilon_1]_\kappa + [ F^\varepsilon_2]_\kappa \big)
\big\}\mathrm{d}s
\bigg\Vert^2_{L^2((0,T)\times \mathbb{R}^3)}
\\&\lesssim
\varepsilon^{1-2\beta}
\mathbb{E} \, 
\bigg\Vert
\int_0^t
S\bigg( \frac{-s}{\varepsilon^{\beta+1}} \bigg)
     \big\{  \mathcal{Q} \mathrm{div}\big([\mathbb{F}^\varepsilon_1]_\kappa + [\mathbb{F}^\varepsilon_2]_\kappa \big)
+
\nabla \big([ F^\varepsilon_1]_\kappa + [ F^\varepsilon_2]_\kappa \big)
\big\} \mathrm{d}s
\bigg\Vert^2_{L^2(\mathbb{R}^3)}.
\end{aligned}
\end{equation}
However, due to Jensen's inequality, the fact that the semigroup is isometric in $L^2$ and the continuity of $\mathcal{Q}$, we obtain
\begin{equation}
\begin{aligned}
\label{estFb}
\mathbb{E} &\, 
\bigg\Vert
\int_0^t
S\bigg( \frac{-s}{\varepsilon^{\beta+1}} \bigg)
     \big\{  \mathcal{Q} \mathrm{div}\big([\mathbb{F}^\varepsilon_1]_\kappa + [\mathbb{F}^\varepsilon_2]_\kappa \big)
+
\nabla \big([ F^\varepsilon_1]_\kappa + [ F^\varepsilon_2]_\kappa \big)
\big\} \mathrm{d}s
\bigg\Vert^2_{L^2(\mathbb{R}^3)}
\\&\lesssim
\mathbb{E} \, 
\int_0^t
\big\Vert
  \mathrm{div}\big([\mathbb{F}^\varepsilon_1]_\kappa + [\mathbb{F}^\varepsilon_2]_\kappa \big)
\big\Vert^2_{L^2(\mathbb{R}^3)}
\mathrm{d}s
+
\mathbb{E} \, 
\int_0^t
\big\Vert
\nabla  \big([ F^\varepsilon_1]_\kappa + [ F^\varepsilon_2]_\kappa \big)
\big\Vert^2_{L^2(\mathbb{R}^3)}
\mathrm{d}s.
\end{aligned}
\end{equation}
If we now choose $s=r=1$ and $p=2$ in \eqref{mollifierB} to estimate $\mathbb{F}^\varepsilon_1$ and $F^\varepsilon_1$ and choose $s=1$ and $p=r=2$ in \eqref{mollifierB} to estimate $\mathbb{F}^\varepsilon_2$ and $F^\varepsilon_2$, 
then it follows from \eqref{estFb} and \eqref{estF3} that
\begin{equation}
\begin{aligned}
\label{estFc}
\mathbb{E} &\, 
\int_0^t
\big\Vert
  \mathrm{div}\big([\mathbb{F}^\varepsilon_1]_\kappa + [\mathbb{F}^\varepsilon_2]_\kappa \big)
\big\Vert^2_{L^2(\mathbb{R}^3)}
\mathrm{d}s
+
\mathbb{E} \, 
\int_0^t
\big\Vert
\nabla \big([ F^\varepsilon_1]_\kappa + [ F^\varepsilon_2]_\kappa \big)
\big\Vert^2_{L^2(\mathbb{R}^3)}
\mathrm{d}s
\\&\lesssim
\kappa^{-\frac{5}{2}}
\mathbb{E} \, 
\int_0^t
\Big(
\Vert
  \mathrm{div}\, 
\mathbb{F}^\varepsilon_1
\Vert^2_{W^{-1,1}(\mathbb{R}^3)}
+
\Vert
\nabla  F^\varepsilon_1
\Vert^2_{W^{-1,1}(\mathbb{R}^3)}
\Big)
\mathrm{d}s
\\&+
\kappa^{-1}
\mathbb{E} \, 
\int_0^t
\Big(
\Vert
  \mathrm{div}\, 
\mathbb{F}^\varepsilon_2
\Vert^2_{W^{-1,2}(\mathbb{R}^3)}
+
\Vert
\nabla  F^\varepsilon_2
\Vert^2_{W^{-1,2}(\mathbb{R}^3)}
\Big)
\mathrm{d}s
\\&\lesssim
\kappa^{-\frac{5}{2}}
\end{aligned}
\end{equation}
for any $\kappa>0$ small.
If  we now choose $\kappa =\varepsilon^{{2\delta\beta}/{5}}$ for any $\delta>0$, then from \eqref{estFa}--\eqref{estFc},
 we have shown that the estimate
\begin{equation}
\begin{aligned}
\label{schrit3}
\mathbb{E} \, 
\Bigg\Vert
\int_0^t
S\bigg( \frac{t-s}{\varepsilon^{\beta+1}} \bigg)
\Bigg\{
&\begin{bmatrix}
       0        \\[0.3em]
       \mathcal{Q} \mathrm{div}\big([\mathbb{F}^\varepsilon_1]_\kappa + [\mathbb{F}^\varepsilon_2]_\kappa \big)
+
\nabla \big([ F^\varepsilon_1]_\kappa + [ F^\varepsilon_2]_\kappa \big)
\end{bmatrix}  
\Bigg\}\mathrm{d}s
\Bigg\Vert^2_{L^2(0,T;L^2(\mathbb{R}^3))}
\\&
\lesssim  \varepsilon^{1-(2+\delta)\beta}
\end{aligned}
\end{equation}
holds uniformly in $\varepsilon>0$.
Now similar to the arguments on \cite[Page 26]{mensah2016existence}, if we let $\mathbf{g}^\varepsilon_i  :=
 \mathbf{g}_i( {\varrho}^\varepsilon , {\varrho}^\varepsilon   {\mathbf{u}}^\varepsilon )$, then we can use It\^o's isometry, properties of the semigroup, the continuity of $\mathcal{Q}$ and a similar argument as in \eqref{schrit2} above to obtain
\begin{equation}
\begin{aligned}
\label{schrit4}
\mathbb{E} &\, 
\Bigg\Vert
\int_0^t 
S\bigg( \frac{t-s}{\varepsilon^{\beta+1}} \bigg)
\begin{bmatrix}
       0        \\[0.3em]
      \mathcal{Q}[\mathbf{G}^\varepsilon]_\kappa
\end{bmatrix}  \mathrm{d}{W}(s)
\Bigg\Vert^2_{L^2(0,T;L^2(\mathbb{R}^3))}
 \lesssim
 \varepsilon^{1-2\beta}
\mathbb{E} \, 
\int_0^t 
\sum_{i\in \mathbb{N}}
\big\Vert
    [\mathbf{g}_i^\varepsilon]_\kappa
\big\Vert^2_{L^2(\mathbb{R}^3)}\mathrm{d}s
\\&
\lesssim
\varepsilon^{1-2\beta}\kappa^{-\frac{5}{2}}
\mathbb{E} \, 
\int_0^t 
\sum_{i\in \mathbb{N}}
\big\Vert
    \mathbf{g}_i^\varepsilon
\big\Vert^2_{W^{-1,1}(K)}\mathrm{d}s
\lesssim
 \varepsilon^{1-(2+\delta)\beta}\,
\mathbb{E} \, 
\int_0^t 
\sum_{i\in \mathbb{N}}
\big\Vert
    \mathbf{g}_i^\varepsilon
\big\Vert^2_{L^1(K)}\mathrm{d}s
\lesssim
 \varepsilon^{1-(2+\delta)\beta}
\end{aligned}
\end{equation}
uniformly in $\varepsilon$ where $K \Subset \mathbb{R}^3$ is the support of the noise, recall \eqref{noiseAssumptionWholespace1}--\eqref{noiseAssumptionWholespace2}. By combining the estimates \eqref{schrit2} and \eqref{schrit3}--\eqref{schrit4}, we get from \eqref{mildAcoustic} that
\begin{equation}
\begin{aligned}
\label{schrit5}
&\mathbb{E} \, 
\Bigg\Vert
\begin{bmatrix}
        [\sigma^\varepsilon ]_\kappa       \\[0.3em]
       \nabla[{\Psi}^\varepsilon]_\kappa
\end{bmatrix} 
\Bigg\Vert^2_{L^2(0,T;L^2(\mathbb{R}^3))}
\lesssim 
\varepsilon^{1-(2+\delta)\beta}
\end{aligned}
\end{equation}
holds uniformly in $\varepsilon>0$ for any $\delta>0$ so in particular,
\begin{equation}
\begin{aligned}
\label{schrit6}
&\mathbb{E} \, \Vert
       \mathcal{Q}[\varrho^\varepsilon\mathbf{u}^\varepsilon]_\kappa
\Vert^2_{L^2(0,T;L^2(\mathbb{R}^3))}
\lesssim 
\varepsilon^{1-(2+\delta)\beta}
\end{aligned}
\end{equation}
holds uniformly in $\varepsilon>0$ since $\nabla \Psi^\varepsilon   = \mathcal{Q} (    \varrho^\varepsilon  \mathbf{u}^\varepsilon )$. Thus, we have shown \eqref{gradMomEst} for $p=2$. The other cases of $p\in[1,2)$ follow from H\"older inequality when one considers a random variable $X$ as the product $1 \times X$.
\\
To show \eqref{gradMomZero}, we use the continuity of $\mathcal{Q}$, the second estimate in Lemma \ref{lem:momAndConvec} and the fact that  $\frac{2\gamma}{\gamma+1}<2$ to obtain for any ball $B_m\subset \mathbb{R}^3$ of radius $m\in \mathbb{N}$,
\begin{equation}
\begin{aligned}
\lim_{\varepsilon\rightarrow0}
\mathbb{E}\bigg[\int_0^T &\big\Vert \mathcal{Q}(\varrho^\varepsilon \mathbf{u}^\varepsilon) \big\Vert_{L^\frac{2\gamma}{\gamma+1}(B_m)}^2   \,\mathrm{d}t\,\bigg]^p 
\lesssim_{p,m}
\lim_{\varepsilon\rightarrow0}
\mathbb{E}\bigg[\int_0^T \big\Vert \mathcal{Q}[\varrho^\varepsilon \mathbf{u}^\varepsilon]_\kappa \big\Vert_{L^2(\mathbb{R}^3)}^2   \,\mathrm{d}t\,\bigg]^p 
\\&+
\lim_{\varepsilon\rightarrow0}
\lim_{\kappa\rightarrow0}
\mathbb{E}\bigg[\int_0^T \Vert [\varrho^\varepsilon \mathbf{u}^\varepsilon]_\kappa
-
(\varrho^\varepsilon \mathbf{u}^\varepsilon) \Vert_{L^\frac{2\gamma}{\gamma+1}(B_m)}^2   \,\mathrm{d}t\,\bigg]^p
\end{aligned}
\end{equation}
The right-hand side converges strongly to zero due to \eqref{gradMomEst} and the second estimate of Lemma \ref{lem:momAndConvec} provided that $\beta<\frac{1}{2+\delta}$ for any $\delta>0$.
\end{proof}

\subsection{Analysis of the gradient part of velocity}
\label{sec:gradVelo}
By relying on the analysis of the gradient part of momenta given in Lemma \ref{lem:GradMomen}, we can now show that the gradient part of the family of velocities also vanishes in the limit.
\begin{lemma}
\label{lem:gradVelocity}
Let $\beta>0$ be a constant. Then for any $\delta>0$, the following  estimates 
\begin{align*}
\mathbb{E}\bigg[\int_0^T \big\Vert\mathcal{Q}( \mathbf{u}^\varepsilon) \big\Vert_{L^2(\mathbb{R}^3)}^2   \,\mathrm{d}t\,\bigg]^p \, 
\lesssim_p \varepsilon^{1-(2+\delta)\beta}, 
\end{align*}
holds uniformly in $\varepsilon>0$ for any $p\in [1,2]$ and
\begin{align}
\label{gradVelZero}
\mathcal{Q}( \mathbf{u}^\varepsilon) \rightarrow 0 \quad \text{in} \quad  L^p \big( \Omega ; L^2\big((0,T) \times\mathbb{R}^3\big)\big)
\end{align}
as $\varepsilon \rightarrow0$ provided that $0<\beta<\frac{1}{2+\delta}$.
\end{lemma}
\begin{proof}
To prove the above lemma, we first note that since
$
\mathbf{u}^\varepsilon=  [\varrho^\varepsilon \mathbf{u}^\varepsilon]_\kappa
- \varepsilon[\sigma^\varepsilon\mathbf{u}^\varepsilon]_\kappa 
+
\mathbf{u}^\varepsilon - [\mathbf{u}^\varepsilon]_\kappa  ,
$
the following inequality 
\begin{align*}
\mathbb{E}\bigg[\int_0^T \Vert\mathcal{Q}( \mathbf{u}^\varepsilon) \Vert_{L^2(\mathbb{R}^3)}^2   \,\mathrm{d}t\,\bigg]^p \, 
\lesssim_p J_1 + J_2 +J_3 
\end{align*}
holds where, thanks to Lemma \ref{lem:GradMomen} the inequality
\begin{align*}
&J_1:= \mathbb{E}\bigg[\int_0^T \big\Vert\mathcal{Q}[ \varrho^\varepsilon\mathbf{u}^\varepsilon ]_\kappa \big\Vert_{L^2(\mathbb{R}^3)}^2   \,\mathrm{d}t\,\bigg]^p
\lesssim_p \varepsilon^{1-(2+\delta)\beta}
\end{align*}
holds uniformly in $\varepsilon>0$ for any $\delta>0$.   By using \eqref{mollifierB}, the second estimate in Lemma \ref{lem:veloAndSomething} and the continuity of $\mathcal{Q}$, we also have 
\begin{align*}
&J_2:= \mathbb{E}\bigg[\int_0^T \big\Vert \varepsilon \mathcal{Q} [ \sigma^\varepsilon\mathbf{u}^\varepsilon ]_\kappa \big\Vert_{L^2(\mathbb{R}^3)}^2   \,\mathrm{d}t\,\bigg]^p
\lesssim \varepsilon^{2p} \,
\kappa^{-2p }
\mathbb{E}\bigg[\int_0^T \Vert \sigma^\varepsilon \mathbf{u}^\varepsilon
\Vert_{W^{-1,2}(\mathbb{R}^3)}^2   \,\mathrm{d}t\,\bigg]^p \lesssim \varepsilon^{2p} \,
\kappa^{-2p }
\end{align*}
uniformly in $\varepsilon>0$. Finally, by using \eqref{mollifierA}, it follows that
\begin{align*}
&J_3:= \mathbb{E}\bigg[\int_0^T \big\Vert\mathcal{Q}( \mathbf{u}^\varepsilon)
-
\mathcal{Q}[ \mathbf{u}^\varepsilon]_\kappa
 \big\Vert_{L^2(\mathbb{R}^3)}^2   \,\mathrm{d}t\,\bigg]^p \lesssim
\kappa^{2p }
\mathbb{E}\bigg[\int_0^T \Vert \nabla \mathbf{u}^\varepsilon
\Vert_{L^2(\mathbb{R}^3)}^2   \,\mathrm{d}t\,\bigg]^p \lesssim
\kappa^{2p }
\end{align*}
holds uniformly in $\varepsilon>0$. By choosing  $\kappa =\varepsilon^{1/2}$ in the estimates for $J_2$ and $J_3$, we have thus shown that  
\begin{align*}
\mathbb{E}\bigg[\int_0^T \Vert\mathcal{Q}( \mathbf{u}^\varepsilon) \Vert_{L^2(\mathbb{R}^3)}^2   \,\mathrm{d}t\,\bigg]^p \, 
\lesssim \varepsilon^{1-(2+\delta)\beta}
\end{align*}
hence our final results.
\end{proof}

\section{Compactness}
\label{subsec:compactness}
We now have all the required estimates so we proceed with showing compactness. For this, we let
\begin{itemize}
\item $\mathcal{L}[ \varrho^\varepsilon]$ be the law of $\varrho^\varepsilon$ on the space $\chi_{\varrho}:=C_w([0,T];L^{\gamma}_{\mathrm{loc}}(\mathbb{R}^3 ))$,
\item $\mathcal{L}[ \mathbf{u}^\varepsilon ]$ be the law of  $ \mathbf{u}^\varepsilon $ on the space $\chi_{ \mathbf{u} } := (L^2(0,T;W^{1,2}( \mathbb{R}^3),w )$,
\item $\mathcal{L}[\varrho^\varepsilon \mathbf{u}^\varepsilon]$ be the law of $\varrho^\varepsilon \mathbf{u}^\varepsilon$ on the space $\chi_{\varrho\bu}:=L^2(0,T;L^{\frac{2\gamma}{\gamma+1}}_{\mathrm{loc}}(\mathbb{R}^3 ))$,
\item $\mathcal{L}[ W]$ be the law of $W$ on the space $\chi_{W}:=C ([0,T];\mathfrak{U}_0 ) $
\end{itemize}
and let  $\mathcal{L}[ \varrho^\varepsilon,  \bu^\varepsilon, \varrho^\varepsilon \bu^\varepsilon,W]$ be  the joint law of $\chi=\chi_\varrho \times \chi_\bu \times \times_{\varrho\bu}\times \chi_W$.
\begin{lemma}
\label{lem:tightDens}
Let $\beta>0$ be a constant such that for any $\delta>0$, we have $0<\beta<\frac{1}{2+\delta}$. The collection $\{\mathcal{L}[ \varrho^\varepsilon ,\bu^\varepsilon,\varrho^\varepsilon \bu^\varepsilon, W] : \varepsilon \in (0,1)\}$ is tight on $\chi$.
\end{lemma}
\begin{proof}
We start by showing that the collection $\{\mathcal{L}[ \varrho^\varepsilon ] : \varepsilon \in (0,1)\}$ is tight on $\chi_{ \varrho}$. From the continuity equation \eqref{contEqB} and the second uniform bound in Lemma \ref{lem:momAndConvec}, it follows that for any ball $B_m\subset \mathbb{R}^3$ of radius $m\in \mathbb{N}$, there exists a constant $a_m>0$ such that the bound
\begin{align*}
\mathbb{E}\, \bigg[ 
\sup_{t\in [0,T]} \big\Vert \partial_t \varrho^\varepsilon \big\Vert_{W^{-1, \frac{2\gamma}{\gamma+1}} (B_m )}
 \bigg]^p \lesssim a_m
\end{align*} 
holds uniformly in $\varepsilon>0$ for any $p\in[ 1,\infty)$. As such, there exists a Lipschitz continuous representation of the density sequence (not relabelled) satisfying
\begin{align}
\label{holderDensity}
\mathbb{E}\, \big\Vert 
 \varrho^\varepsilon \big\Vert^p_{C^{0,1}([0,T];W^{-1, \frac{2\gamma}{\gamma+1}} (B_m))}
 \lesssim a_m
\end{align} 
uniformly in $\varepsilon>0$. 
Furthermore,  by \cite[Corollary B.2]{ondrejat2010stochastic}, for any sequence of positive numbers $(a_m)_{m\in \mathbb{N}}$ and any sequence of balls $(B_m)_{m\in \mathbb{N}}$, the set
\begin{align*}
K_m= \Big\{ \varrho\in \chi_\varrho \,:\, \Vert \varrho \Vert_{L^\infty(0,T;L^\gamma(B_m))}
+
\Vert \varrho \Vert_{C^{0,1}([0,T];W^{-1,\frac{2\gamma}{\gamma+1}}(B_m))}
\leq a_m, \, m\in \mathbb{N} \Big\}
\end{align*}
is relatively compact in $\chi_\varrho$ for all $\gamma\geq \frac{3}{5}$. If we now take the complement of $K_m$, then we find that there exist an $m\in \mathbb{N}$ such that by the third uniform bound in Lemma \ref{lem:momAndConvec} and \eqref{holderDensity}
\begin{align*}
\mathcal{L}&[\varrho^\varepsilon](K^C_m)= \mathbb{P}\Big(
\Vert \varrho^\varepsilon \Vert_{L^\infty(0,T;L^\gamma(B_m))}
+
\Vert \varrho^\varepsilon \Vert_{C^{0,1}([0,T];W^{-1,\frac{2\gamma}{\gamma+1}}(B_m))}
> a_m
\Big)
\\&\leq
\mathbb{P}\bigg(
\Vert \varrho^\varepsilon \Vert_{L^\infty(0,T;L^\gamma(B_m))}
> \frac{a_m}{2} \bigg)
+
\mathbb{P}\bigg(
\Vert \varrho^\varepsilon \Vert_{C^{0,1}([0,T];W^{-1,\frac{2\gamma}{\gamma+1}}(B_m))}
> \frac{a_m}{2} \bigg)
\\&\leq
\frac{4}{a_m^2}
\mathbb{E}\Big(
\Vert \varrho^\varepsilon \Vert_{L^\infty(0,T;L^\gamma(B_m))}^2
+
\Vert \varrho^\varepsilon \Vert_{C^{0,1}([0,T];W^{-1,\frac{2\gamma}{\gamma+1}}(B_m))}^2
\Big) \lesssim \frac{4}{a_m^2}
\end{align*}
hence the collection $\{\mathcal{L}[ \varrho^\varepsilon ] : \varepsilon \in (0,1)\}$ is tight on $\chi_{ \varrho}$.
\\
Next, to show that the collection $\{\mathcal{L}[ \varrho^\varepsilon \bu^\varepsilon ] : \varepsilon \in (0,1)\}$ is tight on $\chi_{ \varrho\bu}$, we first note that $\varrho^\varepsilon \bu^\varepsilon =\mathcal{P}(\varrho^\varepsilon \bu^\varepsilon) + \mathcal{Q}(\varrho^\varepsilon \bu^\varepsilon)$ where the gradient part of the momenta satisfy \eqref{gradMomZero}. On the other hand, by the continuity of $\mathcal{P}$, the second estimate of Lemma \ref{lem:momAndConvec} and Lemma \ref{lem:SoleMomen}, we can deduce that there exists a constant $a_m>0$ such that for any ball $B_m\subset \mathbb{R}^3$ of radius $m\in \mathbb{N}$, the solenoidal part of the momenta satisfy the following  estimates 
\begin{align*}
\mathbb{E}\bigg[\sup_{t\in[0,T]} \Vert \mathcal{P}(\varrho^\varepsilon \mathbf{u}^\varepsilon) \Vert_{L^\frac{2\gamma}{\gamma+1}(B_m)}   \bigg]^p
+
\mathbb{E} \left\Vert  \mathcal{P}\,( \varrho^\varepsilon\mathbf{u}^\varepsilon) \right\Vert^p_{C^\vartheta\left([0,T];W^{-l,2}(B_m)  \right)}
 \, 
\lesssim_p  a_m
\end{align*}
uniformly in $\varepsilon$. Tightness on $\chi_{ \varrho\bu}$ thus follow from the compact embedding
\begin{align}
\label{cptEmb}
L^\infty \big(0,T;L^\frac{2\gamma}{\gamma+1}(B_m) \big) \cap  C^\vartheta\big([0,T];W^{-l,2}(B_m)\big) \hookrightarrow 
C_w\big([0,T];L^\frac{2\gamma}{\gamma+1}(B_m)\big)
\end{align}
and the fact that $C_w\big([0,T];L^\frac{2\gamma}{\gamma+1}(B_m)\big)$ is contained in $\chi_{ \varrho\bu}$.
\\
Now, using the  first estimate of Lemma \ref{lem:veloAndSomething}, we can conclude that the collection $\{\mathcal{L}[ \bu^\varepsilon ] : \varepsilon \in (0,1)\}$ is tight on $\chi_{ \bu}$ and since $\mathcal{L}[W]$ is a Radon  measure on the Polish space $\chi_W$, it is tight. This completes the proof.
\end{proof}
With the tightness result, Lemma \ref{lem:tightDens} in hand, we can now conclude from the Jakubowski--Skorokhod representation  theorem \cite{jakubowski1998short}, the following result.
\begin{proposition}
\label{prop:Jakubow0}
Let $\beta>0$ be a constant such that for any $\delta>0$, we have $0<\beta<\frac{1}{2+\delta}$.
There exist a complete probability space $(\tilde{\Omega}, \tilde{\mathscr{F}}, \tilde{\mathbb{P}})$ with $\chi$-valued  random variables 
\begin{align*}
(\tilde{\varrho}, \tilde{\mathbf{U}}, \tilde{\mathbf{M}},  \tilde{W})\text{ and }  
(\tilde{\varrho}^\varepsilon, \tilde{\mathbf{U}}^\varepsilon,  \tilde{\mathbf{M}}^\varepsilon, \tilde{W}^\varepsilon)_{\varepsilon \in(0,1)}, 
\end{align*}
such that up to  subsequence (not relabelled)
\begin{enumerate}
\item for all $\varepsilon \in(0,1)$, the joint laws
\begin{align*}
\mathcal{L}[\tilde{\varrho}^\varepsilon, \tilde{\mathbf{U}}^\varepsilon,  \tilde{\mathbf{M}}^\varepsilon, \tilde{W}^\varepsilon]
\text{ and }  
\mathcal{L}[\varrho^\varepsilon,  \mathbf{u}^\varepsilon,  (\varrho^\varepsilon \mathbf{u}^\varepsilon),  W ]
\end{align*}
 coincide on  $\chi$;
\item the law $\mathcal{L}[\tilde{\varrho}, \tilde{\mathbf{U}}, \tilde{\mathbf{M}},  \tilde{W}]$
on $\chi$ is a Radon measure;
\item as  $\varepsilon\rightarrow 0$, we have that the following convergence
\begin{align}
\tilde{\varrho}^{\varepsilon}  \rightarrow \tilde{\varrho} &\quad\text{in}\quad \chi_\varrho
\label{almostSureDensity}
\\
\tilde{\mathbf{U}}^{\varepsilon} \rightarrow \tilde{\mathbf{U}}&\quad\text{in}\quad  \chi_\bu
\label{almostSurePVelocity}
\\
\tilde{\mathbf{M}}^{\varepsilon} \rightarrow \tilde{\mathbf{M}}&\quad\text{in}\quad  \chi_{\varrho
\bu}
\label{almostSurePMomentum}
\\
\tilde{W}^\varepsilon  \rightarrow \tilde{W} &\quad\text{in}\quad \chi_W
\label{almostSureWiener}
\end{align}
holds $\tilde{\mathbb{P}}$-a.s.
\end{enumerate}
\end{proposition}
A direct consequence of the equality of laws as established in Proposition \ref{prop:Jakubow0} is that
\begin{align}
\label{momEqDenVelo}
\tilde{\mathbf{M}}^\varepsilon =\tilde{\varrho}^\varepsilon\tilde{\mathbf{U}}^\varepsilon
\end{align}
holds $\tilde{\mathbb{P}}$-a.s. Also, in combination with Lemma \ref{lem:momAndConvec}, \eqref{gradVelZero} and \eqref{gradMomZero}, we can deduce that the following convergence
\begin{align}
\tilde{\varrho}^{\varepsilon}  \rightarrow 1 &\quad\text{in}\quad C_w \big([0,T];L^{\gamma}_{\mathrm{loc}}(\mathbb{R}^3 )\big),
\label{almostSureDensityA}
\\
\mathcal{Q}( \tilde{\mathbf{U}}^\varepsilon) \rightarrow 0 &\quad \text{in} \quad  L^2\big((0,T) \times\mathbb{R}^3\big),
\label{almostSurePVelocityA}
\\
\mathcal{Q}(\tilde{\mathbf{M}}^{\varepsilon}) \rightarrow 0&\quad\text{in}\quad  L^2\big(0,T;L^{\frac{2\gamma}{\gamma+1}}_{\mathrm{loc}}(\mathbb{R}^3 ) \big),
\label{almostSurePGradMomentumA}
%\\
%\tilde{W}^\varepsilon  \rightarrow \tilde{W} &\quad\text{in}\quad C ([0,T];\mathfrak{U}_0 )
%\label{almostSureWienerA}
\end{align}
holds $\tilde{\mathbb{P}}$-a.s. as $\varepsilon \rightarrow0$.
Furthermore, we can show the following result.
\begin{lemma}
\label{lem:lastLem}
Let $\beta>0$ be a constant such that for any $\delta>0$, we have $0<\beta<\frac{1}{2+\delta}$. The following convergence
\begin{align}
\mathcal{P}(\tilde{\mathbf{M}}^{\varepsilon}) \rightarrow \tilde{\mathbf{M}}&\quad\text{in}\quad  L^2\big(0,T;L^{\frac{2\gamma}{\gamma+1}}_{\mathrm{loc}}(\mathbb{R}^3 ) \big),
\label{almostSurePSolMomentumA}
\\
\mathcal{P}(\tilde{\mathbf{U}}^{\varepsilon}) \rightarrow \tilde{\mathbf{M}}&\quad\text{in}\quad  L^2\big(0,T;L^{\frac{2\gamma}{\gamma+1}}_{\mathrm{loc}}(\mathbb{R}^3 ) \big)
\label{almostSurePSolVelocityA}
\end{align}
holds $\tilde{\mathbb{P}}$-a.s. as $\varepsilon \rightarrow0$ and finally,
\begin{align}
\label{momEqVelo}
\tilde{\mathbf{U}}=\tilde{\mathbf{M}}
\end{align}
holds $\tilde{\mathbb{P}}$-a.s.
\end{lemma}
\begin{proof}
The first result \eqref{almostSurePSolMomentumA} is an immediate consequence of \eqref{almostSurePMomentum} and  \eqref{almostSurePGradMomentumA}. The second result \eqref{almostSurePSolVelocityA} is however slightly trickier since $\chi_\bu$ in \eqref{almostSurePVelocity} is only endowed with the weak topology. Nevertheless, since equality of laws holds true, by observing that
\begin{align*}
\mathcal{P}(\tilde{\mathbf{U}}^{\varepsilon} ) = \mathcal{Q}(\tilde{\varrho}^\varepsilon\tilde{\mathbf{U}}^{\varepsilon}) + \mathcal{P}(\tilde{\varrho}^\varepsilon\tilde{\mathbf{U}}^{\varepsilon}
)- \mathcal{Q}(\tilde{\mathbf{U}}^{\varepsilon} )  - \varepsilon \,\tilde{\sigma}^\varepsilon\tilde{\mathbf{U}}^{\varepsilon} =:I_1 +I_2 +I_3+I_4
\end{align*}
where $\tilde{\sigma}^\varepsilon=\frac{\tilde{\varrho}^\varepsilon-1}{\varepsilon}$, we deduce from \eqref{plusEst1}--\eqref{plusEst2} (noting that  $\frac{2\gamma}{\gamma+1} < \frac{4}{3}$ when $\gamma<2$) that
\begin{align}
I_4 \rightarrow 0 &\quad\text{in}\quad  L^2\big(0,T;L^{\frac{2\gamma}{\gamma+1}}_{\mathrm{loc}}(\mathbb{R}^3 ) \big)
\end{align}
holds $\tilde{\mathbb{P}}$-a.s. as $\varepsilon \rightarrow0$. The convergence \eqref{almostSurePSolVelocityA} therefore follow from 
 \eqref{momEqDenVelo}, \eqref{almostSurePVelocityA}, \eqref{almostSurePGradMomentumA} and  \eqref{almostSurePSolMomentumA}. Finally, \eqref{momEqVelo} follow from \eqref{almostSurePMomentum} and \eqref{momEqDenVelo} as well as the fact that from \eqref{almostSurePVelocity} \eqref{almostSureDensityA}, one can deduce that
\begin{align}
\tilde{\varrho}^\varepsilon\tilde{\mathbf{U}}^{\varepsilon} \rightharpoonup \tilde{\mathbf{U}}&\quad\text{in}\quad  L^1\big(0,T;L^1_{\mathrm{loc}}(\mathbb{R}^3 ) \big)
\end{align}
holds $\tilde{\mathbb{P}}$-a.s.
\end{proof}
\section{Identification of the limit system}
\label{sec:limitWholeSpace}
On the new probability space $(\tilde{\Omega}, \tilde{\mathscr{F}}, \tilde{\mathbb{P}})$,  it follows from \cite[Theorem 2.9.1]{breit2018stoch} that $\tilde{W}^\varepsilon=\sum_{k\in \mathbb{N}} e_k \tilde{\beta}_k^\varepsilon$ is a cylindrical Wiener process with respect to the following filtration
\begin{align*}
\tilde{\mathscr{F}}_t^\varepsilon= \sigma\left( \sigma_t[\tilde{\varrho}^\varepsilon] \cup \sigma_t[\tilde{\mathbf{U}}^\varepsilon] \cup \bigcup_{k\in \mathbb{N}} \sigma_t[\tilde{\beta}_k^\varepsilon])   \right), \quad t\in[0,T]
\end{align*}
for all $\varepsilon>0$, thus, $\tilde{\mathscr{F}}_t^\varepsilon$ is non-anticipative with respect to $\tilde{W}^\varepsilon$. By using \cite[Lemma  2.9.3]{breit2018stoch} and Proposition \ref{prop:Jakubow0}, we can pass to the limit $\varepsilon \rightarrow0$ to conclude that
\begin{align*}
\tilde{\mathscr{F}}_t &= \sigma\left( \sigma_t[\tilde{\varrho}] \cup \sigma_t[\tilde{\mathbf{U}}] \cup \bigcup_{k\in \mathbb{N}} \sigma_t[\tilde{\beta}_k])   \right).
\end{align*} 
is non-anticipative with respect to $\tilde{W}$. It therefore follow that $\tilde{W}$ is a cylindrical Wiener process with respect to $(\tilde{\mathscr{F}}_t )_{t\geq0}$ by virtue of \cite[Lemma 2.1.35, Corollary 2.1.36]{breit2018stoch}.\\
Furthermore, due to the equality of laws from Proposition \ref{prop:Jakubow0}, in combination with \cite[Theorem 2.9.1]{breit2018stoch} (see also \cite[Theorem 2.4.31]{mensah2019theses}), we can conclude with the following.
\begin{corollary}
\label{cor:tildeWeakSol}
Let $\varepsilon>0$ and $\varepsilon^\beta \Delta \tilde{V}^{\varepsilon} = \tilde{\varrho}^\varepsilon -1$ where $\beta>0$. Then
$
[(\tilde{\Omega} , \tilde{\mathscr{F}} , (\tilde{\mathscr{F}}_t^\varepsilon)_{t\geq0}, \tilde{\mathbb{P}} ) , \tilde{\varrho}^{\varepsilon} , \tilde{\mathbf{U}}^{\varepsilon}, \tilde{V}^{\varepsilon}, \tilde{W}^\varepsilon]
$
is a finite energy weak martingale solution of \eqref{contEqB}--\eqref{elecFieldB} in the sense of Definition \ref{def:martSolutionExis} with initial law $\tilde{\Lambda}^\varepsilon$.
\end{corollary}
We now have all the tools to identify the limit in the weak formulation of our system. We will only show in detail, how to deal with  the deterministic part of the momentum equation since  the identification of the continuity equation is quite standard and can be found in  \cite{breit2015incompressible} and \cite{mensah2016existence} while for  the stochastic integral, we will show  just the main steps.

%Note the extra density in front of the noise term in \eqref{momEqB} can be estimated by its averaged value and thus  we can indeed use the  exact same argument of the proof of \cite[Lemma 11]{mensah2016existence}.
\begin{lemma}
Let $\beta>0$ be a constant such that for any $\delta>0$, we have $0<\beta<\frac{1}{2+\delta}$, for all $t\in [0,T]$ and $\bm{\phi} \in C^\infty_{c, \mathrm{div}}(\mathbb{R}^3)$, let
\begin{align*}
M(\varrho, \bu, V)_t &:=\int_{\mathbb{R}^3}\varrho \bu(t) \cdot \bm{\phi}\dx
-
\int_{\mathbb{R}^3}\varrho\bu(0) \cdot \bm{\phi}\dx
-
\int_0^t\int_{\mathbb{R}^3}\varrho \bu \otimes \bu : \nabla \bm{\phi}\dx \mathrm{d}s
\\&+
\nu_1
\int_0^t\int_{\mathbb{R}^3}\nabla\bu :\nabla \bm{\phi}\dx \mathrm{d}s
-
\int_0^t\int_{\mathbb{R}^3}\bigg(\frac{\varrho -1}{\varepsilon^2}\nabla V\bigg) \cdot \bm{\phi}\dx \mathrm{d}s.
\end{align*}
Then $M(\tilde{\varrho}^\varepsilon, \tilde{\mathbf{U}}^\varepsilon, \tilde{V}^\varepsilon)_t \rightarrow M(1, \tilde{\mathbf{U}}, \tilde{V})_t$ $\tilde{\mathbb{P}}$-a.s. as $ \varepsilon \rightarrow 0$. 
%provided that $0<\beta<\frac{1}{2+\delta}$ for any $\delta>0$.
\end{lemma}
\begin{proof}
To begin with, we use \eqref{almostSurePVelocity} and \eqref{almostSureDensityA} to conclude that
for all $t\in [0,T]$ and $\bm{\phi} \in C^\infty_{c, \mathrm{div}}(\mathbb{R}^3)$,
\begin{align*}
&\int_{\mathbb{R}^3}\tilde{\varrho}^\varepsilon \tilde{\mathbf{U}}^\varepsilon(t) \cdot \bm{\phi}\dx
-
\int_{\mathbb{R}^3}\tilde{\varrho}^\varepsilon \tilde{\mathbf{U}}^\varepsilon(0) \cdot \bm{\phi}\dx
+
\nu_1
\int_0^t\int_{\mathbb{R}^3}\nabla \tilde{\mathbf{U}}^\varepsilon  : \nabla \bm{\phi}\dx \mathrm{d}s
\\&\rightarrow
\int_{\mathbb{R}^3}\tilde{\varrho} \tilde{\mathbf{U}}(t) \cdot \bm{\phi}\dx
-
\int_{\mathbb{R}^3}\tilde{\varrho}\tilde{\mathbf{U}}(0) \cdot \bm{\phi}\dx
+
\nu_1
\int_0^t\int_{\mathbb{R}^3} \nabla \tilde{\mathbf{U}}  : \nabla \bm{\phi}\dx \mathrm{d}s
\end{align*}
$\tilde{\mathbb{P}}$-a.s. as $\varepsilon \rightarrow 0$.
Now since we have 
\begin{align*}
\tilde{\varrho}^\varepsilon \tilde{\mathbf{U}}^\varepsilon \otimes \tilde{\mathbf{U}}^\varepsilon =\mathcal{P}\tilde{\mathbf{U}}^\varepsilon \otimes \mathcal{P}\tilde{\mathbf{U}}^\varepsilon +  (\tilde{\varrho}^\varepsilon-1)\tilde{\mathbf{U}}^\varepsilon\otimes \tilde{\mathbf{U}}^\varepsilon + \mathcal{P}\tilde{\mathbf{U}}^\varepsilon \otimes \mathcal{Q}\tilde{\mathbf{U}}^\varepsilon + \mathcal{Q}\tilde{\mathbf{U}}^\varepsilon \otimes \mathcal{P}\tilde{\mathbf{U}}^\varepsilon + \mathcal{Q}\tilde{\mathbf{U}}^\varepsilon \otimes \mathcal{Q} \tilde{\mathbf{U}}^\varepsilon,
\end{align*}
it follows from \eqref{almostSureDensityA}, \eqref{almostSurePVelocityA}, \eqref{almostSurePSolVelocityA} and \eqref{momEqVelo} that for all $t\in [0,T]$ and $\bm{\phi} \in C^\infty_{c, \mathrm{div}}(\mathbb{R}^3)$,
\begin{align*}
\int_0^t\int_{\mathbb{R}^3}\tilde{\varrho}^\varepsilon \tilde{\mathbf{U}}^\varepsilon \otimes \tilde{\mathbf{U}}^\varepsilon  : \nabla \bm{\phi}\dx \mathrm{d}s
\rightarrow
\int_0^t\int_{\mathbb{R}^3} \tilde{\mathbf{U}} \otimes \tilde{\mathbf{U}}  : \nabla \bm{\phi}\dx \mathrm{d}s
\end{align*}
$\tilde{\mathbb{P}}$-a.s. as $\varepsilon \rightarrow 0$.
\\
For the last term of $M(\varrho, \bu, V)_t $ which includes the potential $V$ in the convergence analysis we have to distinguish between $\gamma<2$ and $\gamma\geq 2$.\\
If $\gamma\in (\frac{3}{2}, 2)$,  as a result of Corollary \ref{cor:tildeWeakSol}, we can conclude from \eqref{orlizSigma} that  
\begin{align}
\tilde{\mathbb{E}}\bigg[\sup_{t\in[0,T]} \Vert  \tilde{\sigma}^\varepsilon \Vert_{L^{\gamma}(\mathbb{R}^3)}^\gamma   \bigg]^p \, 
\lesssim_p  1
\label{orlizSigmab}
\end{align}
holds uniformly in $\varepsilon>0$ for any $p\in[1,\infty)$ where
\begin{align*}
\tilde{\sigma}^\varepsilon=\frac{\tilde{\varrho}^\varepsilon-1}{\varepsilon}.
\end{align*}
Now since
\begin{align*}
\sup_{t\in[0,T]} \big\Vert \varepsilon^\frac{\beta-2}{2}\nabla\tilde{V}^\varepsilon  \big\Vert_{L^2(\mathbb{R}^3)}
\lesssim
\sup_{t\in[0,T]} \big\Vert \varepsilon^\frac{\beta-2}{2}\nabla\tilde{V}^\varepsilon  \big\Vert_{L^2(\mathbb{R}^3)}^2
=
\sup_{t\in[0,T]} \big\Vert \varepsilon^{\beta-2}\vert\nabla\tilde{V}^\varepsilon \vert^2  \big\Vert_{L^1(\mathbb{R}^3)},
\end{align*}
and
\begin{align*}
\sup_{t\in[0,T]} \Vert \tilde{\varrho}^\varepsilon-1\Vert_{L^\gamma(\mathbb{R}^3)}
\lesssim
\sup_{t\in[0,T]} \Vert \tilde{\varrho}^\varepsilon-1\Vert_{L^\gamma(\mathbb{R}^3)}^\gamma
=
\varepsilon^\gamma\,
\sup_{t\in[0,T]} \Vert \tilde{\sigma}^\varepsilon \Vert_{L^\gamma(\mathbb{R}^3)}^\gamma,
\end{align*}
holds uniformly in $\varepsilon>0$, from  \eqref{velocityl6} and \eqref{orlizSigmab} and  it follows that if $\gamma\in (\frac{3}{2}, 2)$, then
\begin{equation}
\begin{aligned}
\sup_{t\in[0,T]} \bigg\Vert 
\frac{\tilde{\varrho}^\varepsilon -1}{\varepsilon^2}\nabla \tilde{V}^\varepsilon \bigg\Vert_{L^\frac{2\gamma}{\gamma+2}(\mathbb{R}^3)}
\lesssim\varepsilon^{\gamma-\beta/2 -1} \rightarrow 0
\end{aligned}
\end{equation}
$\tilde{\mathbb{P}}$-a.s. as $\varepsilon\rightarrow0$ provided that
\begin{align}
\label{betaGammaMinusOne}
\beta < 2(\gamma-1), \quad \frac{3}{2}<\gamma< 2.
\end{align}
On the other hand, if $\gamma\geq2$, then we obtain for  any $2<q<6$ and its H\"older conjugate $q'>0$,
\begin{equation}
\begin{aligned}
\label{epsPlusEps}
\bigg\vert
\int_0^t&\int_{\mathbb{R}^3}\bigg(\frac{\tilde{\varrho}^\varepsilon -1}{\varepsilon^2}\nabla \tilde{V}^\varepsilon\bigg) \cdot \bm{\phi}\dx \mathrm{d}s
\bigg\vert
\lesssim
\varepsilon^{-1}
\sup_{t\in[0,T]} \big\Vert\nabla\tilde{V}^\varepsilon\cdot \bm{\phi}  \big\Vert_{W^{1,q'}(\mathbb{R}^3)}
\, \Vert \tilde{\sigma}^\varepsilon \Vert_{L^2(0,T;W^{-1,q}(\mathbb{R}^3))}
\\&
\lesssim
\varepsilon^{-1}
\Vert \bm{\phi} \Vert_{W^{1,\frac{2q}{q-2}}(\mathbb{R}^3)}
\bigg(
\sup_{t\in[0,T]} \big\Vert\nabla\tilde{V}^\varepsilon \big\Vert_{L^2(\mathbb{R}^3)}
+
\sup_{t\in[0,T]} \big\Vert\Delta\tilde{V}^\varepsilon \big\Vert_{L^2(\mathbb{R}^3)}
\bigg)
 \Vert \tilde{\sigma}^\varepsilon \Vert_{L^2((0,T)\times \mathbb{R}^3)}
\end{aligned}
\end{equation}
$\tilde{\mathbb{P}}$-a.s.  for all $t\in [0,T]$ and $\bm{\phi} \in C^\infty_{c, \mathrm{div}}(\mathbb{R}^3)$.
However, by using the Poisson equation \eqref{elecFieldB}, the estimate \eqref{orlizSigmaNot} and the bottom left estimate of \eqref{energyEst1}, we obtain 
\begin{equation}
\begin{aligned}
\label{epsPlusEps1}
\sup_{t\in[0,T]} \big\Vert\nabla\tilde{V}^\varepsilon \big\Vert_{L^2(\mathbb{R}^3)}
+
\sup_{t\in[0,T]} \big\Vert\Delta\tilde{V}^\varepsilon \big\Vert_{L^2(\mathbb{R}^3)}
 \lesssim \varepsilon^\frac{2-\beta}{2}+ \varepsilon^{1-\beta}
\end{aligned}
\end{equation}
uniformly in $\varepsilon>0$. Now by using the equality of laws given by Proposition \ref{prop:Jakubow0} and Sobolev embedding in time, we can conclude from \eqref{orlizSigmaNot} that $\tilde{\sigma}^\varepsilon$ is $\tilde{\mathbb{P}}$-a.s. square-integrable in spacetime and thus
\begin{align}
\label{convolutionConv}
[\tilde{\sigma}^\varepsilon]_\kappa \rightarrow \tilde{\sigma}^\varepsilon \quad \text{in} \quad L^2((0,T)\times \mathbb{R}^3)
\end{align}
$\tilde{\mathbb{P}}$-a.s. as $\kappa\rightarrow0$. On the other hand, by \eqref{schrit5}, $[\tilde{\sigma}^\varepsilon]_\kappa$ can be bounded  in terms of $\varepsilon>0$. Therefore by writing $\tilde{\sigma}^\varepsilon =[\tilde{\sigma}^\varepsilon]_\kappa + (\tilde{\sigma}^\varepsilon-[\tilde{\sigma}^\varepsilon]_\kappa)$, we can conclude by using the equality of laws given by Proposition \ref{prop:Jakubow0}, \eqref{convolutionConv} and the acoustic estimate \eqref{schrit5}, that $\tilde{\mathbb{P}}$-a.s.,
\begin{align}
\label{oneMinusTwoBeta}
\Vert \tilde{\sigma}^\varepsilon \Vert_{L^2((0,T)\times\mathbb{R}^3)}\lesssim 
\varepsilon^{1-(2+\delta)\beta}
\end{align}
holds uniformly in $\varepsilon>0$ for any $\delta> 0$. If we now substitute \eqref{oneMinusTwoBeta} and \eqref{epsPlusEps1} into \eqref{epsPlusEps}, then $\tilde{\mathbb{P}}$-a.s. ,
\begin{equation}
%\begin{aligned}
\label{epsPlusEps2}
\bigg\vert
\int_0^t\int_{\mathbb{R}^3}\bigg(\frac{\tilde{\varrho}^\varepsilon -1}{\varepsilon^2}\nabla \tilde{V}^\varepsilon\bigg) \cdot \bm{\phi}\dx \mathrm{d}s
\bigg\vert
\lesssim
\varepsilon^{1-\beta(5/2+\delta)}+ \varepsilon^{1-\beta(3+\delta)} \lesssim
\varepsilon^{1-\beta(3+\delta)}
%\end{aligned}
\end{equation}
holds uniformly in $\varepsilon>0$ when $\gamma\geq2$. In this case, convergence to zero as $\varepsilon\rightarrow0$ holds true provided that 
\begin{align}
\beta < \frac{1}{3+\delta}.
\label{betaGammaMinusOneb}
\end{align}
Since  $0<\beta<\frac{1}{2+\delta}$ for any $\delta>0$ the two conditions \eqref{betaGammaMinusOne}, \eqref{betaGammaMinusOneb} are satisfied  and 
%$\gamma\in(\frac{3}{2},2)$ in \eqref{betaGammaMinusOne}, for the following bounds
%\begin{align}
%\label{betaLessOneThird}
%\beta < \frac{1}{2+\delta}< 2(\gamma-1),
%\end{align}
our final result holds true.
\end{proof}
To pass to the limit in the stochastic integral, 
we need the following analogous version of \cite[Lemma 2.6.6]{breit2018stoch} on the whole space, see \cite[Lemma 2.4.35]{mensah2019theses}. We refer to the former for the proof.
\begin{lemma} 
\label{lem:noiseProduct}
Let $\mathfrak{U}\subset\mathfrak{U}_0$ be a separable Hilbert space and let $\big(\Omega, \mathscr{F}, \mathbb{P} \big)$ be a complete probability space. For $n\in \mathbb{N}$, let $W_n $ be an $\big( \mathscr{F}_t^n\big)$-cylindrical Wiener process such that
\begin{align*}
W_n \rightarrow W \quad \text{in} \quad C\big([0,T]; \mathfrak{U}_0 \big)\quad \text{in probability}
\end{align*}
with $W=\sum_{k\in \mathbb{N}} e_k \beta_k$. Also, for each  $n\in \mathbb{N}$, let $\Phi_n$ be an $( \mathscr{F}_t^n)$-progressively measurable stochastic process belonging to $L^2 \big(0,T; L_2\big(\mathfrak{U};W^{l,2}(\mathbb{R}^3) \big) \big)$ $\mathbb{P}$-a.s. for some $l\in \mathbb{R}$ and for which,
\begin{align*}
\Phi_n \rightarrow \Phi \quad \text{in} \quad L^2 \big(0,T; L_2\big(\mathfrak{U};W^{l,2}(\mathbb{R}^3) \big) \big) \quad \text{in probability}.
\end{align*}
Then after possible change on a measure zero set in $\Omega \times (0,T)$, we gain 
\begin{align*}
\int_0^\cdot \Phi_n \, \mathrm{d}W_n \rightarrow  \int_0^\cdot \Phi \, \mathrm{d}W
\quad \text{in} \quad L^2\big(0,T;W^{l,2}(\mathbb{R}^3) \big) \quad \text{in probability}
\end{align*}
and that $\Phi$ is a progressively measurable process with respect to the following filtration.
\begin{align*}
\sigma  \bigg( \bigcup_{k=1}^\infty \sigma_t \big[\Phi e_k \big] \cup \sigma_t\big[ \beta_k \big] \bigg)
\end{align*}
where in the above, $(\beta_k)_{k\in \mathbb{N}}$ is a family of mutually independent Brownian motions and $(e_k)_{k\in \mathbb{N}}$ are orthonormal basis of a separable Hilbert space $\mathfrak{U}$.
\end{lemma}
Moving on, we combine the assumptions on the noise term \eqref{noiseAssumptionWholespace1}--\eqref{noiseAssumptionWholespace2} with the limits \eqref{almostSureDensityA}--\eqref{almostSurePGradMomentumA} (also recall \eqref{momEqDenVelo}) and Lemma \ref{lem:lastLem} which allow us to obtain the convergence
\begin{align}
\label{limitnoiseTilde}
\mathbf{G}(\tilde{\varrho}^\varepsilon, \tilde{\varrho}^\varepsilon\tilde{\mathbf{U}}^\varepsilon) 
\rightarrow
\mathbf{G}(1,\tilde{\mathbf{U}}^\varepsilon)
\quad \text{in} \quad
L^2\big( 0,T;L_2(\mathfrak{U}; W^{-l,2}(\mathbb{R}^3))\big)
\end{align}
$\tilde{\mathbb{P}}$-a.s. for some $l>0$. For the details, we refer the reader to \cite[Page 294]{breit2018stoch} which was done for a more restrictive assumption on the noise coefficient due to the lack of compactness for the gradient part of the velocity field. Nevertheless, an analogous analysis holds true for the less restrictive noise assumption \eqref{noiseAssumptionWholespace2} as already illustrated in the proof of \cite[Lemma 11]{mensah2016existence}.
If we then combine \eqref{limitnoiseTilde} with the convergence \eqref{almostSureWiener} for $\tilde{W}^\varepsilon$, we are able to pass to the limit in the stochastic integral by virtue of  Lemma \ref{lem:noiseProduct}.
\\ 
By collecting the various results above, we can finally conclude with the following result which ends the proof of the main Theorem \ref{thm:wholespace}
\begin{theorem}
\label{thm:conclusionWholeSpace}
The following $[(\tilde{\Omega},\tilde{\mathscr{F}},(\tilde{\mathscr{F}}_t)_{t\geq0}, \tilde{\mathbb{P}}); \tilde{\mathbf{U}}, \tilde{W}]$
is a weak martingale solution of \eqref{incomprSPDE} on $\mathbb{R}^3$ in the sense of  Definition \ref{def:martSolutionIncompre} with the initial law $\Lambda$.
\end{theorem}

\section{Related singular limits: the zero-electron-mass limit}
\label{appendix}
In this section we analyse a further singular limit  related to  plasma physics which, although it has a different physical meaning compared to  the previous one, can  be handled, from a mathematical point of view, with the same tools developed in the previous sections. We recall that  a plasma consists of a negatively charged elections $q^e=-1$ of mass $m^e$, a positively charged ion $q^i=+1$ of mass $m^i$ and a negligible amount of neutral particles. If we let $\varrho^e=\varrho^e(t,x)\in \mathbb{R}_{\geq 0}$ and $\bu^e=\bu^e(t,x)\in \mathbb{R}^3$ (respectively, $\varrho^i=\varrho^i(t,x)\in \mathbb{R}_{\geq 0}$ and $\bu^i=\bu^i(t,x)\in \mathbb{R}^3$) be the scaled density and mean velocity of the electron (respectively, ion) and we let $V=V(t,x)$ be the scaled electric potential, then on the macroscopic level, a viscous Newtonian plasma subject to random forces satisfy the following (suitably scaled) damped Navier--Stokes--Poisson system 
\begin{align}
&\mathrm{d} \varrho^\alpha + \mathrm{div}(\varrho^\alpha \mathbf{u}^\alpha) \, \mathrm{d}t =0, 
\label{contEqC}
\\
&\mathrm{d}(\varrho^\alpha \mathbf{u}^\alpha) +
\bigg[
\mathrm{div} (\varrho^\alpha \mathbf{u}^\alpha \otimes \mathbf{u}^\alpha) +\frac{1}{\delta^\alpha} \nabla p  (\varrho^\alpha) \bigg]\, \mathrm{d}t 
=\big[ \nu_1 \Delta \mathbf{u}^\alpha + (\nu_2 + \nu_1)\nabla \mathrm{div} \mathbf{u}^\alpha
\nonumber\\
&\qquad\qquad\qquad\qquad\qquad
 + \frac{q^\alpha}{\delta^\alpha} {\varrho^\alpha} \nabla V \big] \, \mathrm{d}t
+ \frac{1}{\tau^\alpha}\varrho^\alpha\mathbf{u}^\alpha
\, \mathrm{d}t
 +\frac{1}{\tau^\alpha}  \mathbf{G}(\varrho^\alpha,  \varrho^\alpha\mathbf{u}^\alpha) \, \mathrm{d}W^\alpha,
\label{momEqC}\\
&\lambda^2 \Delta V = \varrho^e -\varrho^i 
\label{elecFieldC}
\end{align}
where $\alpha=e,i$ are indexes standing for electrons and ions, $\tau^\alpha>0$ are the scaled relaxation times, $\lambda>0$ is the Debye length and 
\begin{align}
p(\varrho^\alpha)=a(\varrho^\alpha)^\gamma
\end{align}
is the isentropic pressure with $a>0$ 
% being the square reciprocal of the Mach number 
and $\gamma> \frac{3}{2}$ being the adiabatic exponent. In addition,  $\delta^\alpha = \frac{m^\alpha v^2_0}{k_BT_0}$ is a dimensionless parameter where $k_B$ is the Boltzmann constant and  $v_0$ and $T_0$ are the typical velocity and temperature values respectively for the plasma. See \cite{jungel2000hierarchy, jungel2001hierarchy} for the deterministic counterpart of the above system and \cite[Section 1.1]{mensah2019theses} for the modelling of stochastic fluids.
\\
Since it is well-known that the electron mass $m^e$ is far smaller than the ion mass $m^i$, i.e., $m^e\ll m^i$, movements of ions are far limited with respect to movements of electrons on a finite observation time. In other words, on a time scale that one observes the first `reasonable' movement in the electrons, the ions are essentially stationary. This difference in time scales offers justification to decouple the system and study separately, the evolution of ions or that of the electrons. Moreover, since the parameter $\delta^\alpha$ is the ratio between the electron and the ion mass (for details see \cite{jungel2000hierarchy}) then the limit $\delta^\alpha\to 0$ makes sense and, since we assume that the ion density is given, this limit goes under the name of {\em zero electron mass limit}.

\subsection{Evolution of the electrons}
Hence we consider the ion density given and we concentrate on just the electrons so that the superscript $\alpha=e$. Since the ions and electrons are coupled only through the Poisson equation, for simplicity, we set $\varrho^i=1$ as our background charged ion. We also assume that the relaxation times $\tau^\alpha$ are fixed and that the plasma is not subject to drift. With these assumptions in hand, for  $\varepsilon>0$  small, we set the dimensionless parameter $\delta^e$, which is proportional to the electron mass $m^e$, to $\varepsilon^2$ i.e. $\delta^e=\varepsilon^2$. We also set $\lambda=a=\tau=1$, $\varrho^e=\varrho^\varepsilon$, $\bu^e=\bu^\varepsilon$, $V=V^\varepsilon$, $W^e=W$ and analyse the following system
\begin{align}
&\mathrm{d} \varrho^\varepsilon + \mathrm{div}(\varrho^\varepsilon \mathbf{u}^\varepsilon) \, \mathrm{d}t =0, 
\label{contEqD}\\
&\mathrm{d}(\varrho^\varepsilon \mathbf{u}^\varepsilon) +
\big[
\mathrm{div} (\varrho^\varepsilon \mathbf{u}^\varepsilon \otimes \mathbf{u}^\varepsilon) + \varepsilon^{-2} \nabla  (\varrho^\varepsilon)^\gamma \big]\, \mathrm{d}t 
=\big[ \nu_1 \Delta \mathbf{u}^\varepsilon + (\nu_2 + \nu_1)\nabla \mathrm{div} \mathbf{u}^\varepsilon
\nonumber\\
&\qquad\qquad\qquad\qquad\qquad\qquad\qquad
- \varepsilon^{-2}\varrho^\varepsilon \nabla V \big] \, \mathrm{d}t
 + \mathbf{G}(\varrho^\varepsilon,  \varrho^\varepsilon\mathbf{u}^\varepsilon) \, \mathrm{d}W,
\label{momEqD}\\
& \Delta V = \varrho^\varepsilon -1, 
\label{elecFieldD}
\end{align}
as $\varepsilon\rightarrow0$.
\\
For completness we mention that an equivalent formulation of the system \eqref{contEqD}--\eqref{elecFieldD} is given by rescaling the electric potential 
 as $V=\varepsilon^2\widetilde{V}$ and reads as follows,
\begin{align}
&\mathrm{d} \varrho^\varepsilon + \mathrm{div}(\varrho^\varepsilon \mathbf{u}^\varepsilon) \, \mathrm{d}t =0, 
\label{contEqDb}\\
&\mathrm{d}(\varrho^\varepsilon \mathbf{u}^\varepsilon) +
\big[
\mathrm{div} (\varrho^\varepsilon \mathbf{u}^\varepsilon \otimes \mathbf{u}^\varepsilon) + \varepsilon^{-2} \nabla  (\varrho^\varepsilon)^\gamma \big]\, \mathrm{d}t 
=\big[ \nu_1 \Delta \mathbf{u}^\varepsilon + (\nu_2 + \nu_1)\nabla \mathrm{div} \mathbf{u}^\varepsilon
\nonumber\\
&\qquad\qquad\qquad\qquad\qquad\qquad\qquad
- \varrho^\varepsilon \nabla  \widetilde{V} \big] \, \mathrm{d}t
 + \mathbf{G}(\varrho^\varepsilon,  \varrho^\varepsilon\mathbf{u}^\varepsilon) \, \mathrm{d}W,
\label{momEqDb}\\
& \varepsilon^{2}\Delta \widetilde{V} = \varrho^\varepsilon -1. 
\label{elecFieldDb}
\end{align}
We observe that from a purely mathematical point of view, the system \eqref{contEqD}--\eqref{elecFieldD} corresponds to \eqref{contEqB}--\eqref{elecFieldB} when $\beta=0$. This observation can be made rigorous.\\
Indeed, by adapting the  proof of  Proposition \ref{prop:expBound} to the case $\beta=0$, we can prove the following result for the acoustic system.
\begin{proposition}
\label{nonsingularKG}
For any $p\in[1,\infty)$ and  $r\in [1,\infty]$, the solution pair $(\nabla\Psi^\varepsilon, \sigma^\varepsilon)$ of the system
\begin{equation}
\begin{aligned}
\mathrm{d}\sigma^\varepsilon +  \Delta {\Psi}^\varepsilon \,\mathrm{d}t   =  0,&   \\
 \mathrm{d}\nabla{\Psi}^\varepsilon  +  \big(\gamma \nabla - \nabla \Delta^{-1} \big) \sigma^\varepsilon\,\mathrm{d}t   =    0,&
 \\
\sigma^\varepsilon (0) = \sigma^\varepsilon_0 ; \quad \nabla{\Psi}^\varepsilon(0) = \nabla{\Psi}^\varepsilon_0&
\end{aligned}
\end{equation}
satisfy the estimate
\begin{align*}
&\mathbb{E} \, \Vert  \nabla \Psi^\varepsilon \Vert_{ L^r\big(0,T;L^2(\mathbb{R}^3) \big)}^p
+
\mathbb{E} \, \Vert  \sigma^\varepsilon \Vert_{ L^r\big(0,T;L^2(\mathbb{R}^3) \big)}^p 
\lesssim
\mathbb{E} \, \Vert \sigma_0^\varepsilon \Vert_{L^2(\mathbb{R}^3)}^p
+
\mathbb{E} \, \Vert \nabla \Psi_0^\varepsilon \Vert_{L^2(\mathbb{R}^3)}^p
\end{align*}
uniformly in $\varepsilon$ for $\sigma_0^\varepsilon \in L^p(\Omega; L^2(\mathbb{R}^3))$ and $\nabla \Psi_0^\varepsilon \in L^p(\Omega; L^2(\mathbb{R}^3))$.
\end{proposition}
Then, by following the same line of arguments of the proof of the Theorem \ref{thm:wholespace}, we  obtain the stochastic case the following zero-electron-mass number limit result (for the deterministic setting compare with \cite{DoFeNo12}). 
\begin{theorem}
\label{thm:electronMass}
Let  $\Lambda$ be a given Borel probability measure on $L^2_{\mathrm{div}}(\mathbb{R}^3)$. For $\varepsilon >0$ and  $\gamma>3/2$, we let $\Lambda^\varepsilon$ be a family of Borel probability measures on $\big[ L^1_x \big]^2= L^1(\mathbb{R}^3) \times L^1(\mathbb{R}^3)$ such that
\begin{equation}
\begin{aligned}
\Lambda^\varepsilon  \Big\{ (\varrho, \mathbf{m})  \, : \,   
 \vert \varrho-1 \vert\leq  \varepsilon M  \Big\}=1
\end{aligned}
\end{equation}
holds for a deterministic constant $M>0$ which is independent of $\varepsilon>0$. For all $p\in[1,\infty)$, we assume that the following moment estimate
\begin{equation}
\begin{aligned}
\int_{[ L^1_x]^2}  \left\Vert \frac{1}{2}\frac{\vert\mathbf{m}\vert^2}{\varrho}  +  \frac{1}{\varepsilon^2} H(\varrho,1) 
+
\frac{1}{\varepsilon^2} \vert \nabla V \vert^2 \right\Vert^p_{L^1_x}   \mathrm{d}\Lambda^\varepsilon(\varrho,  \mathbf{m}) \lesssim_{p} 1,
\end{aligned}
\end{equation}
holds uniformly in $\varepsilon$.
Further assume that \eqref{noiseAssumptionWholespace1}-- \eqref{noiseAssumptionWholespace2} holds and that  the marginal law of $\Lambda^\varepsilon$ corresponding to the second component converges to $\Lambda$ weakly in the sense of measures on $L^{\frac{2\gamma}{\gamma+1}}(\mathbb{R}^3 )$. If
\begin{align}
\left[(\Omega^\varepsilon,\mathscr{F}^\varepsilon,(\mathscr{F}^\varepsilon_t),\mathbb{P}^\varepsilon);\varrho^\varepsilon, \mathbf{u}^\varepsilon, V^\varepsilon, W^\varepsilon  \right]
\end{align}
is a finite energy weak martingale solution of \eqref{contEqD}--\eqref{elecFieldD} in the sense of Definition \ref{def:martSolutionExis} with initial law $\Lambda^\varepsilon$, then
\begin{align*}
\varrho^\varepsilon \rightarrow 1\quad&\text{in law in}\quad L^\infty(0,T;L^{\gamma}_{\mathrm{loc}} (\mathbb{R}^3 )),
\\
 \mathbf{u}^\varepsilon \rightarrow  \mathbf{U}\quad&\text{in law in}\quad 
  \big(L^2(0,T;W^{1,2}(\mathbb{R}^3 )),w \big)
\end{align*}
as $\varepsilon \rightarrow0$ and  $\mathbf{U}$ is a weak martingale solution of \eqref{incomprSPDE}  in the sense of  Definition \ref{def:martSolutionIncompre} with  initial law $\Lambda$.
\end{theorem}

%****List of advanced bibliographystyle****
% 1. spbasic
% 2. spphys
% 3. spmpsci

%\bibliographystyle{spmpsci}
%\bibliography{myBibliography}

\end{document}